\documentclass[5p,twocolumn]{elsarticle}

\usepackage[utf8]{inputenc}
\usepackage{graphicx}
\usepackage{amsmath,amssymb,amsfonts,bm}
\usepackage{subcaption}
\usepackage[utf8]{inputenc}
\usepackage{ntheorem}
\usepackage{microtype}
\usepackage{enumitem}
\usepackage{xcolor}
\usepackage{accents}
\usepackage{hyperref}
\usepackage{bbm}

\hypersetup{
    colorlinks=true,
    linkcolor=blue,
    filecolor=magenta,      
    urlcolor=cyan,
    }
\theoremseparator{:}
\newtheorem{defi}{\indent Definition}

\newtheorem{theo}{\indent Theorem}
\newtheorem{Lemm}{\indent Lemma}

\newtheorem{Assum}{\indent Assumption}
 
\newtheorem{rem}{\indent Remark}
\newenvironment{pf}{\textit{Proof:}}{}

\newcommand{\juan}[1]{\textcolor{black}{#1}}

\begin{document}

\begin{frontmatter}

\title{A Mode-wise Adaptive Observer for Autonomous Switched Nonlinear Systems\tnoteref{t1}} 
\tnotetext[t1]{The material in this paper was not presented at any conference.}

\author[UoI]{Juan Sereno}

\author[UoI,IMCI]{E.A. Hernandez-Vargas\corref{cor1}}
\ead{esteban@uidaho.edu}

\cortext[cor1]{Corresponding author}

\address[UoI]{Department of Mathematics and Statistical Science, University of Idaho, Moscow, Idaho, 83844-1103, USA}
\address[IMCI]{Institute for Modeling Collaboration and Innovation (IMCI), University of Idaho, Moscow, Idaho, USA}
  

\begin{abstract}
In the digital twin paradigm, online parameter updating is essential\juan{. Unlike} a static model, a digital twin must continuously adapt to the evolving dynamics of the system it represents. Adaptive observers, which jointly estimate states and parameters from online data, are therefore an increasingly important tool. In this work, we formulate a mode-wise adaptive observer for a class of autonomous switched nonlinear systems with switched unknown parameters. The main challenges are twofold: removing the disturbance that switching injects into the parameter-error dynamics and guaranteeing persistence of excitation over a finite-time window without relying on an input signal. To address them, we assign a dedicated adaptive observer to each mode, active only on its corresponding interval, which directly removes the zero-input disturbance caused by switching. We then introduce a finite-window persistence-of-excitation condition together with a minimum dwell-time condition, under which the parameter estimation error is contractive and the state estimation error is bounded within the active interval of each mode. \juan{The performance of the proposed approach is illustrated with an academic example and with a practical example of evolutionary therapies.}
\end{abstract}

\end{frontmatter}

\section{Introduction}\label{s:introduction} 

Digital twins are virtual replicas that evolve alongside their physical counterparts and have moved from manufacturing origins~\cite{grieves16} into the medical field, where they claim improved patient-specific prediction and therapy design~\cite{NASEM23,bjornsson19,laubenbacher24,niarakis24}. A well-defined feature of a digital twin is online adaptation: the computational model must proactively track the ever-evolving dynamics of the physical twin~\cite{NASEM23}. Therefore, parameter updating is central to the paradigm. Current practice, however, relies heavily on batch-based identification, which delays the update until enough input-output data have been collected~\cite{paoletti07}. In the adaptive observer domain, one can find an appealing alternative since it performs joint and online state-parameter estimation; although it has not been thoroughly investigated within the digital twin framework~\cite{rodriguez25,sereno25i}.

The joint state-parameter estimation problem has been addressed mainly through two different avenues. In one hand, stochastic filtering converts the unknown parameters as augmented states and propagates them through Bayesian recursions, as in the extended, unscented, and ensemble Kalman filters~\cite{simon06,chen03,van01,evensen09,deng13,afshari17}; these methods are flexible and widely used, but their convergence guarantees are generally asymptotic and statistical in nature. Deterministic observer-based methods, by contrast, aim at provable convergence of the joint estimate~\cite{afri16,ortega21,marino92}. Adaptive observers belong to this second front: under suitable observability-identifiability and some class of excitation conditions, they deliver exponential parameter convergence with explicit error bounds, which suits well within the calibration framework of digital twins.

The theoretical foundations of adaptive observers can be traced back to the seminal works~\cite{kreisselmeier74,Luders73,Bastin88,ioannou96,ortega89}. In \cite{zhang02a}, an adaptive observer for linear time-varying systems is proposed, imposing conditions on the observability matrix and the persistency of excitation (PE) condition for the regressor to guarantee exponential convergence. The state-affine setting that we adopt in this work was developed in~\cite{hammouri90,besancon00,besancon06}, while extensions to nonlinearly parameterized systems can be found in~\cite{farza09,tyukin13,farza18}. \juan{In~\cite{ortega15}, the authors translate the reconstruction of the states into the estimation of the constant initial conditions of a suitably transformed system. This idea was later extended in~\cite{ortega21} and, more recently, to state-affine systems with uncertain parameters and unknown additive output disturbances~\cite{romero25}.} The role of PE in securing parameter convergence is itself classical~\cite{Boyd1986,Mareels1988,Narendra2012}, and recent efforts have sought to relax it toward initial- or interval-excitation (IE) conditions~\cite{aranovskiy15,tomei22}. \juan{The notion of excitation on a single finite interval goes back to~\cite{kreisselmeier90}, where richness is required only over a bounded time window rather than uniformly along the whole trajectory.} \juan{Recently, \cite{ortega26} proved that IE of the regressor is a sufficient condition for a standard least-squares estimator, augmented with one algebraic step, to attain finite convergence time.} 
\juan{Applications of these estimators span diverse domains as energy or healthcare, see~\cite{zhao14,wang14,rodriguez25}.}

\juan{Switched systems have been proved to be a solid foundation for modeling and control in a broad spectrum of applications \cite{hernandez11,deaecto10,lee08}.} Several identification strategies transform the switched dynamics into an AutoRegressive eXogenous (ARX) form~\cite{paoletti07,borchersen15,bencherki25}; these strategies focus on recovering the switching behavior, i.e., switching instants, number of subsystems, and switching signal reconstruction, but not the parameters of each mode. Dedicated switching observers have also been designed for specific estimation problems~\cite{aguado18,aguado21}. A separate line builds on the Dynamic Regressor Extension and Mixing (DREM) technique~\cite{aranovskiy16}, which relaxes the classical PE condition and decouples the estimation problem to enable element-wise parameter adaptation; \cite{pyrkin19} developed a DREM-based observer for nonlinear systems with constant parameters, and~\cite{liu23} extended these ideas to systems with switched unknown parameters by exploiting the zero-input responses to build the regressor. 

\juan{Despite these contributions, joint state and parameter estimation for switched systems remains overlooked~\cite{ortega22,glushchenko23,liu23,romero23}. In~\cite{liu23}, a DREM-based adaptive observer is designed for an input-driven nonlinear switched system with constant matrices, the switching induced zero-input response is treated by augmenting the parameter vector with the states at the switching instants (SASI) and taking advantage of the element-wise adaptation of the DREM technique to achieve asymptotic convergence. While in~\cite{romero23}, the authors combines a generalized parameter estimation-based observer with a least-squares plus DREM estimator to guarantee finite convergence time (FCT) within each sub-interval of a linear time-varying system at the cost of resetting part of the observer at every switch. Our proposal differs on three matters. First, the switched system is autonomous and nonlinear, so the switching signal is the sole source of excitation. Second, rather than augmenting the parameter vector or resetting a scheme, we suppress the zero-input disturbance structurally by assigning one adaptive observer per mode, active only on its own active interval. Finally, we consider a finite-window persistence-of-excitation condition and the dwell-time constraint to Produce a a contractive parameter estimation error explicitly governed by the length of the active interval of each subsystem. In a medical digital twin context, this allow for quantify how long a therapy (or mode), must be sustained for model recalibration.
}

Beyond estimation, the analysis, stabilization, and control of switched systems form a very active research line in which dwell-time constraints play a central role~\cite{liberzon03}. Recent contributions include the dwell-time stabilization of switched affine systems via Lyapunov--Metzler inequalities~\cite{russo25}, observer-based sampled-data stabilization under dwell-time constraints~\cite{katz26,cai26}, and the characterization of permanence regions and control-invariant sets for waiting-time switched systems~\cite{perez22,perez23,perez25}. These temporal notions are closely related to the dwell-time condition that we impose on the switching signal to guarantee the contractiveness on parameter estimation error.

The objective of this work is to develop a mode-wise adaptive observer for joint state and parameter estimation in autonomous switched nonlinear systems, where a switching signal, rather than an external input, is the sole driver of the dynamics. In contrast to \cite{liu23,romero23}, our design removes the zero-input disturbance caused by switching without relying on a control input. We \juan{exemplify} the approach on an HIV within-host model under a switched therapy schedule \cite{hernandez12,hernandez13s,sereno25}, extending adaptive observers into switched biological systems. 

\juan{\textbf{Notation.}}
\juan{Throughout, $\mathbb{R}$, $\mathbb{R}_{+}$, $\mathbb{N}$, and $\mathbb{N}_{0}$ denote the real numbers, the positive reals, the natural numbers, and the natural numbers including zero, respectively. $\mathbb{R}^{n}$ denotes the vector space of all $n-$tuples of real numbers. $\mathbb{R}^{n \times m}$ denotes the space of $n \times m$ matrices with real entries. For a given vector $v\in\mathbb{R}^{n}$ or matrix $A\in\mathbb{R}^{n\times m}$, $(\cdot)^{\top}$ denotes transposition, $\|v\|$ the Euclidean norm, $\|A\|$ the induced $2$-norm, and $\lambda_{\min}(A)$ denotes the smallest eigenvalue. $I_{n}$ is the identity matrix of $n \times n$ dimension. For a symmetric matrix $P \in \mathbb{R}^{n \times n}$, $P\succ0$ (respectively, $P\succeq0$) indicates that $P$ is positive definite (respectively, positive semidefinite). For $\Gamma\succ0$, the square weighted norm is $\|v\|^{2}_{\Gamma}:=v^{\top}\Gamma\,v$. Accent $\hat{(\cdot)}$ denotes an estimate, $\tilde{(\cdot)}$ the corresponding estimation error (i.e., $\tilde{x}\triangleq\hat{x}-x$), $\bar{(\cdot)}$ and $\underaccent{\bar}{(\cdot)}$ denotes an upper and lower bound, respectively.
}

\section{Background} \label{s:math_modeling}

Consider an autonomous nonlinear switched system given by,
\begin{subequations}\label{eq:asnl}
\begin{align}
\dot{x}(t) &= f_{\sigma} \left( x(t), \theta_{\sigma(t)} \right), \label{eq:asnl_state} \\
y(t) &= h\left( x(t)\right), \label{eq:asnl_output} \\
x(0) &= x_{0}, \label{eq:asnl_X0}
\end{align}
\end{subequations}
where $x \in \mathcal{X} \subset \mathbb{R}^{n_{x}}$ stand for the state vector and $y \in \mathcal{Y} \subset \mathbb{R}^{n_{y}}$ stand for the measurable output, with both $\mathcal{X}$, and $\mathcal{Y}$ being compact sets. The switching system (\ref{eq:asnl}) is assumed to have a finite number of subsystems (or modes), denoted as $n_{s} \in \mathbb{N}$, and the switching signal $\sigma(t)$ determines which subsystem is activated by setting the switched parameter vector $\theta_{\sigma(t)} \in \Theta \triangleq \{\theta_1, \theta_2, \ldots, \theta_{n_s}\}$, where $\theta_{i} \in \mathbb{R}^{n_p}$. \juan{Here, the switched signal is defined as follows.}
\juan{
\begin{defi}[Switching signal] \label{def:sw_signal}
    Consider the switched system~(\ref{eq:asnl}). The map,
    \begin{equation} 
    \sigma(t) : [0,\infty) \to \mathcal{M}, \quad \forall ~ t \in \mathbb{R}_{+}, \quad \text{with} ~ \mathcal{M} \subset \mathbb{N}, 
    \label{eq:sw_signal}
    \end{equation}
    is called the switching signal for~(\ref{eq:asnl}) if, for an ordered sequence of switching instants, denoted as $\{t_{k}\}_{k \in \mathbb{N}_{0}}$, with $t_{0} = 0$ and $t_{k} < t_{k+1}$, the switching signal $\sigma(t)$ is a known piecewise constant function such that $\sigma(t) = \ell$ for all $t \in [t_{k}, t_{k+1})$, where $\ell \in \mathcal{M}$ stand for a specific selected subsystem (or mode), and $\mathcal{M}$ stands for the finite set of available subsystems (or modes), $\mathcal{M} \triangleq \{1, 2, \ldots, n_s\}$. Accordingly, $\sigma(t)$ holds the property of parameter selection, with $\theta_{\sigma(t)} = \theta_{\ell}$, and $\dot{\theta}_{\ell}=0$, for $t \in [t_{k}, t_{k+1})$, in the switched system ~(\ref{eq:asnl}).
\end{defi}
}
Given a switching signal as in Definition~(\ref{def:sw_signal}), let $\sigma(t) = \ell$ on $[t_{k}, t_{k+1})$; then the system dynamics of~(\ref{eq:asnl}) are governed by
\begin{equation} 
\dot{x} = f_{\ell}\left( x, \theta_{\ell} \right), \quad \text{for} ~ t \in [t_k, t_{k+1}), \quad \text{with} ~ \ell \in \mathcal{M},
\label{eq:f_ell}
\end{equation}
where $f_{\ell}: \mathcal{X}\times\Theta_{\ell} \rightarrow \mathcal{X}$ is a sufficiently smooth self-map that represents the autonomous dynamics of mode $\ell$. For every realization $\sigma(t) = \ell$, we assume that the trajectory $x(t)$ remains in the compact set $\mathcal{X}$ and the output $y(t)$ in the compact set $\mathcal{Y}$, for all $t \in [t_k, t_{k+1})$. For many applications, the switching signal $\sigma(t)$ stands as the sole control action, representing an external manipulation, such as drugs/therapies in medicine, that allows reaching a desired control objective, such as an outcome for a particular patient \cite{sereno25}. 
In this context, the time interval $[t_k, t_{k+1})$ is crucial, as it represents the duration for which the current mode $\ell \in \mathcal{M}$ remains active and, therefore, the time interval during which the adaptive observer is able to estimate the $\theta_{\ell}$ parameters. Consequently, the active time is defined below.
\juan{
\begin{defi}[Active time] \label{def:Active_time}
    Consider the switched system~(\ref{eq:asnl}) under the switching signal~(\ref{eq:sw_signal}). Let $\sigma(t) = \ell$, with $\ell \in \mathcal{M}$, be the active mode during the time interval $[t_k, t_{k+1})$, with $k \in \mathbb{N}_{0}$. The active time interval of the subsystem (or mode)~$\ell$ denoted as $\mathrm{AT}_{\ell}$, is the elapsed time $\mathrm{AT}_{\ell} \triangleq t_{k+1} - t_k$, satisfying  $\mathrm{AT}_{\ell} > 0$.   
\end{defi}
}
\juan{The active time, as in Definition~(\ref{def:Active_time}), corresponds to the actual length of the time window $[t_k, t_{k+1})$ from which a specific active mode~$\ell$ remains active. Here, we consider a minimal duration constraint, named the dwell-time condition, on this active time interval as follows.}
\juan{
\begin{defi}[Dwell time] \label{def:dwell_time}
    Consider the switched system~(\ref{eq:asnl}) under the switching signal~(\ref{eq:sw_signal}). Let $\sigma(t) = \ell$, with $\ell \in \mathcal{M}$, be the active mode, which holds an active time $\mathrm{AT}_{\ell}$ as expressed in Definition~(\ref{def:Active_time}). The active time of the subsystem (or mode)~$\ell$ is subject to dwell-time constraints, if $\mathrm{AT}_{\ell}\geq T_{\textrm{min},\ell}$, for all $k \in \mathbb{N}_{0}$, with $T_{\textrm{min},\ell} > 0$.
\end{defi}
}
\juan{Noteworthy, while the active time corresponds to the period during which the subsystem remains active, the dwell-time condition is a minimum length constraint imposed on the switching signal realization that allows to guarantee the properties of the mode-wise adaptive observer.
} Under Definition~(\ref{def:sw_signal}), (\ref{def:Active_time}), and (\ref{def:dwell_time}), the joint state and switched-parameter estimation problem is stated as follows.

\smallskip
\noindent\textbf{Problem Formulation.}
Given an autonomous switched nonlinear system~(\ref{eq:asnl}), under a switching signal $\sigma(t)$ (as in Definition~(\ref{def:sw_signal})) and subject to dwell time constraints (as in Definition~(\ref{def:dwell_time})), design a subfamily of adaptive observers
%
\begin{equation}\label{eq:problem_observers}
    \Pi \;\triangleq\; \{\Pi_{\ell}\}_{\ell\in\mathcal{M}},
    \qquad
    \Pi_{\ell}:\;\mathcal{M} \times \mathcal{Y}\;\longrightarrow\;\mathcal{X}\times\Theta_{\ell},
\end{equation}
such that, for every active time interval $[t_k, t_{k+1})$ where $\sigma(t)=\ell$ (as in Definition~(\ref{def:Active_time})), the observer $\Pi_{\ell}$ produces a joint state-parameter estimate $\bigl(\hat{x}_{\ell}(t),\hat{\theta}_{\ell} (t)\bigr)$ from causal output data $\{y(s)\}_{s\in[t_{k},t]}$, while every inactive observer $\Pi_{\ell'}$, $\ell'\neq\ell$, holds its internal state. 

\smallskip
\subsection*{Work contributions}
While existing literature has already explored adaptive observer architectures for switched unknown parameters \cite{liu23,romero23}, our approach distinguishes itself by considering the case of autonomous switched nonlinear systems. Besides, in our proposal we address the problem of zero-input disturbance by tailoring a dedicated adaptive observer to each subsystem during its active interval. Lastly, we consider the assumption of finite-window persistence of excitation along with the dwell-time constraint for switched system to explicitly derive an equation that connects the parameter-bounded error contraction with the length of the active time interval of each subsystem. Finally, instead of aiming for finite-time convergence our proposal focuses on guaranteeing boundedness in the parameter and state estimation error, which may be valuable for complex computational models within the medical digital twin framework.

\begin{figure}[ht]
\centering
\includegraphics[width=1.0\columnwidth]{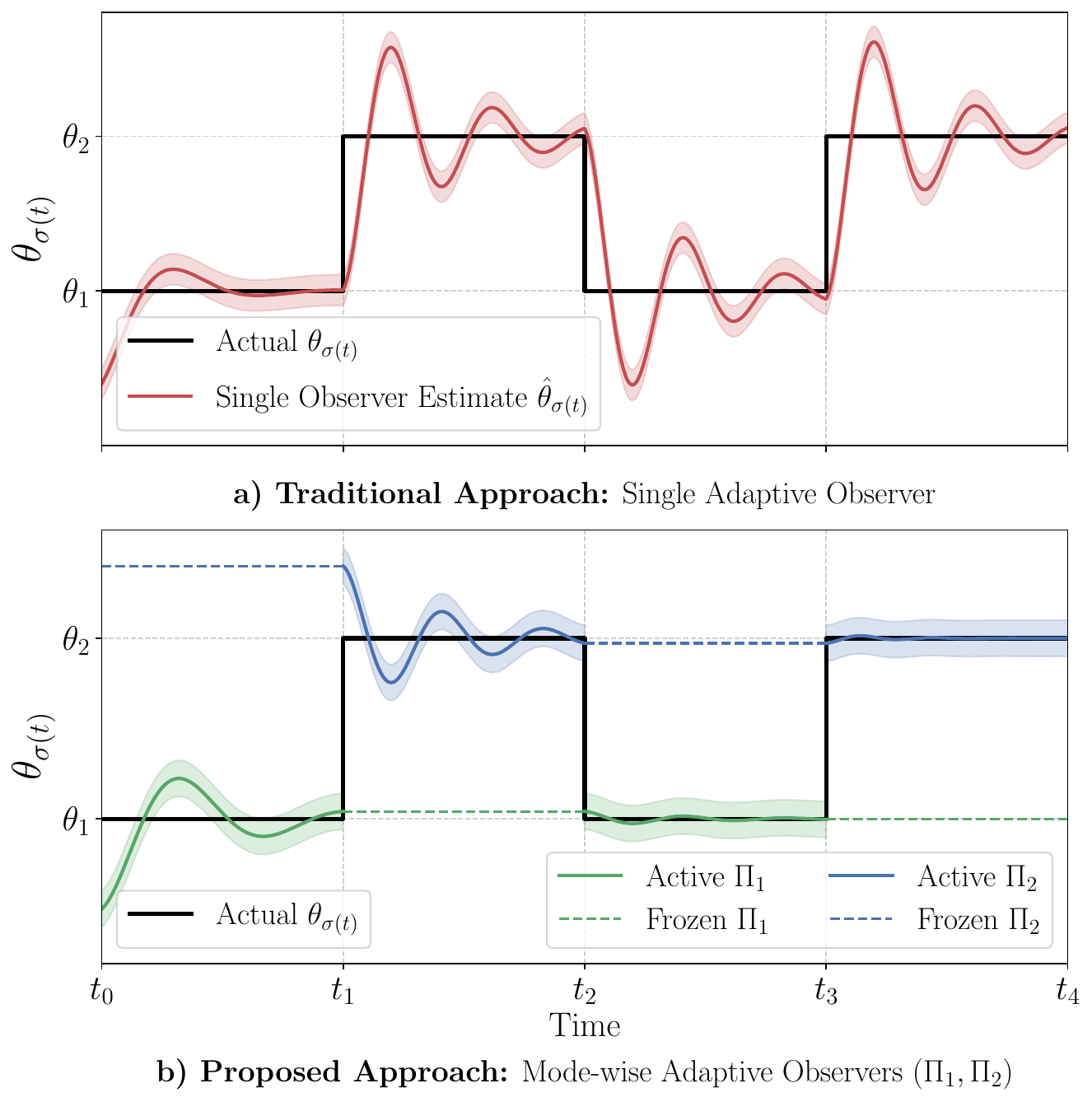}
	\caption{Mode-wise adaptive observer: illustrative scenario. Top graph shows a single adaptive observer applied for switched unknown parameter estimation; the black and red lines represent the actual and estimation values, respectively. Bottom graph shows the mode-wise approach $\Pi \;\triangleq\; \{\Pi_{1}, \Pi_{2}\}$, the green and blue lines represent the $\Pi_{1}$, and $\Pi_{2}$, respectively. Shadow areas represent confident intervals.}
	\label{fig:sigma}
\end{figure}

Figure~\ref{fig:sigma} illustrates the problem formulation for a scenario with two modes, $\sigma(t)\in\{1,2\}$, and switching instants $\{t_{1},t_{2},t_{3},t_{4}\}$. The top panel depicts the traditional single-observer approach. As one observer must track the parameter vector $\theta_{\sigma(t)}$ that changes at every switch, each switching instant injects a new error into the estimation dynamics that is never fully removed, producing the permanent bias visible in the estimate $\hat{\theta}(t)$. The bottom panel depicts the proposed mode-wise scheme, in which a dedicated observer is assigned to each mode and is active only on its corresponding interval. During $[t_{0},t_{1})$ we have $\sigma(t)=1$, so $\Pi_{1}$ is active and refines $\hat{\theta}_{1}$, although it may not reach steady state within a single active interval of duration $\mathrm{AT}_{1}$. At $t_{1}$ the signal switches to $\sigma(t)=2$: $\Pi_{1}$ freezes, holding its current estimate, while $\Pi_{2}$ activates from a tailored initial condition and estimates $\theta_{2}$ on $[t_{1},t_{2})$. At $t_{2}$ the signal returns to $\sigma(t)=1$: $\Pi_{2}$ freezes and $\Pi_{1}$ resumes from the value it held at $t_{1}$, continuing to refine $\hat{\theta}_{1}$. This activate--freeze--resume cycle repeats at each subsequent switch. Because every observer sees only its own mode, the zero-input disturbance that the switching would otherwise inject is removed, and each $\tilde{\theta}_{\ell}$ decreases across the, possibly non-contiguous, intervals on which mode $\ell$ is active.

\section{\juan{Preliminaries}} \label{s:Preliminaries}

In this section, we first present a remark on the observability and identifiability properties for each member function $f_{\ell}$, for all the subsystems $\ell \in \mathcal{M}$ \cite{xia03,villaverde19a}. 

\subsection{Observability analysis}

To access the observability of $f_{\ell}$, let us consider first the Lie derivatives definition as stated in \cite{nijmeijer90}.

\begin{defi}[Lie derivatives] \label{def:LieDeriva}
    The Lie Derivative of $h(x)$ with respect to $f_{\ell}(x, \theta_{\ell})$ is computed as
    \begin{equation}
        L_f h(x) = \frac{\partial h(x)}{\partial x} f_{\ell}(x, \theta_{\ell}),
    \end{equation}
    and, higher-order Lie derivatives can be computed recursively by
    \begin{equation}
        L_f^i h(x) = \frac{\partial L_f^{i-1} h(x)}{\partial x} f_{\ell}(x, \theta_{\ell}).
    \end{equation}
\end{defi}

Then, the observation space is defined as follows.

\begin{defi}[Observability matrix] \label{def:Obsmatrix}
    Consider the autonomous function $f_{\ell}$ given by (\ref{eq:f_ell}). The nonlinear observability matrix $\mathcal{O}_{\ell}$ of the subsystem $\ell \in \mathcal{M}$ can be generated by $h(x)$ and its Lie derivatives \cite{nijmeijer90} as follows:
\begin{equation} 
        \mathcal{O}_{\ell}(x) = \begin{pmatrix}
\frac{\partial}{\partial x} h(x) \\[10pt]
\frac{\partial}{\partial x} \left(L_f h(x)\right) \\[10pt]
\frac{\partial}{\partial x} \left(L_f^2 h(x)\right) \\[10pt]
\vdots \\[10pt]
\frac{\partial}{\partial x} \left(L_f^{n_{x}-1} h(x)\right)
\end{pmatrix},
\label{eq:Obsmatrix}
\end{equation}
in which $n_{x}$ stands for the state dimension of the switched system~(\ref{eq:asnl}).
\end{defi}

From Definition~(\ref{def:Obsmatrix}), the following Lemma is derived.
\begin{Lemm}[Local nonlinear observability] \label{Lemm:obs_f_ell}
    The autonomous function $f_{\ell}$ given by (\ref{eq:f_ell}), derived from the switched realization $\sigma(t) = \ell$, is locally observable at $x_0$, with $x_0 \neq 0$.
\end{Lemm}
The proof of Lemma~(\ref{Lemm:obs_f_ell}) for the switched systems employed in this work was performed by computing $\text{rank}(\mathcal{O}_{\ell}(x_0))$ according to Theorem $3.32$ in \cite{nijmeijer90}. It is important to remark that the observability rank condition is a sufficient, but not strictly necessary, condition for nonlinear observability \cite{villaverde19a}. Moreover, a nontrivial computational cost is regarding Lie derivatives that need to be determined in order to test $\mathcal{O}_{\ell}$'s rank, since full rank may be achieved without computing all the $n_x - 1$ derivatives (from which the rank of $\mathcal{O}_{\ell}$ can no longer increase); the reader is referred to \cite{villaverde16,villaverde19b} for a more detailed discussion.

\subsection{Identifiability analysis}

The property of structural identifiability in the context of this work regards the possibility of giving unique values to the unknown subsystem parameters $\theta_{\ell}$ from available measurable outputs, Equation~(\ref{eq:asnl_output}), assuming noise-free and continuous-time longitudinal data. In the following, $y(x_0, \theta_{\ell})$ denotes the autonomous output map of the subsystem $f_{\ell}$ with respect to parameter set $\theta_{\ell}$. The structural local identifiability for one given component parameter $\theta_{i \ell} \in \theta_{\ell} \triangleq [\theta_{1 \ell}, \ldots, \theta_{n_p \ell}]^{T}$, with $i \in \{1, \ldots, n_p \}$, of the subsystem~(\ref{eq:f_ell}) is defined as follows.
\begin{defi} \label{def:param_Ident}
    \cite{walter97}. A parameter $\theta_i$ is structurally locally identifiable if for almost all $\theta \in \mathbb{R}_{+}$, there exists a neighborhood  $v(\theta)$ of $\theta$, such that:
    \begin{equation}
        \theta^{*} \in v(\theta), \quad y(\mathbf{x}_0, \theta) =   y(\mathbf{x}_0, \theta^{*}) \Rightarrow \theta^{*}_{i} = \theta_{i}.
    \end{equation}
\end{defi}
The following definition accounts for the general structural identifiability property of every component of the parameter vector $\theta_{\ell}$ of the subsystem~(\ref{eq:f_ell}).
\begin{defi} \label{def:model_Ident}
    \cite{walter97}. The subsystem~(\ref{eq:f_ell}) is structurally locally identifiable if all the parameters $\theta_{i \ell} \in \theta_{\ell}$, with $i \in \{1, \ldots, n_p \}$, are structurally locally identifiable.
\end{defi}
From Definition~(\ref{def:model_Ident}), the following Lemma is derived.
\begin{Lemm}[Local structural identifiability] \label{Lemm:ident_switched}
    The subsystem $f_{\ell}(x, \theta_{\ell})$ given by (\ref{eq:f_ell}) with a switched realization $\sigma(t) = \ell$, $\ell \in \mathcal{M}$, is structurally locally identifiable at $\theta^0_{\ell} \in \mathbb{R}^{n_p}$, with $\theta^0_{\ell} \neq 0$.
\end{Lemm}
The proof of Lemma~(\ref{Lemm:ident_switched}) for the switched systems employed in this work can be formed through several methods proposed in the literature, such as output equality, local state isomorphism, differential algebra, and identifiability tableaus, to mention some \cite{anstett20,chis11}. Another alternative is to test the structural identifiability by augmenting the state vector with the unknown parameters and generalizing an observability-identifiability matrix, as stated in \cite{villaverde16}.
In this work, we test the structural identifiability by augmenting the state vector with the unknown parameters and generalizing the observability-identifiability matrix, as stated in~\cite{villaverde16}.

\section{\juan{Main Results}}

Consider the set of subsystems represented by ~(\ref{eq:f_ell}), where $\sigma(t) = \ell$, $\ell \in \mathcal{M}$, for $t \in [t_k, t_{k+1})$, together with the measurable output (\ref{eq:asnl_output}), if by a change of coordinates \juan{(as in \cite{Rodriguez09})} the subsystem $f_{\ell}(x, \theta_{\ell})$ can be rewritten in a state-affine form linearly parameterized with respect to $\theta_{\ell}$ as,
\begin{equation}
\Sigma_{\ell} : \begin{cases} 
\dot{z}(t) = A(y)z(t) + \beta(y) + \Psi(y)\theta_{\ell}, \\ 
y(t) = C \cdot z(t), \quad \text{with} \quad t \in [t_k,~t_{k+1}).
\end{cases}
\label{eq:stateaffine}
\end{equation}
Where the components of the matrix $A(y)$, $\beta(y)$, and $\Psi(y)$ are uniformly bounded continuous functions depending only on $y$, since $u\equiv 0$ for all the switched modes. Furthermore, if the switched system~(\ref{eq:asnl}) allows it to be rewritten as in (\ref{eq:stateaffine}) for every switched mode $\ell \in \mathcal{M}$, then we can find a more general form that accounts for all the switched modes. Hence, the general state-affine switched form reads as,
\begin{equation}
\Sigma_{\sigma} : \begin{cases} 
\dot{z}(t) = A(y)z(t) + \beta(y) + \Psi(y)\theta_{\sigma(t)}, \\ 
y(t) = C \cdot z(t), \quad \text{with} \quad t \geq 0.
\end{cases}
\label{eq:swtiched_state_affine}
\end{equation}

For the joint estimation of states and parameters in autonomous switched systems like~(\ref{eq:asnl}) that are suitable to be rewritten in a state-affine switched form as stated by~(\ref{eq:swtiched_state_affine}), a family of adaptive observers can be designed. Let us first consider the following \juan{two} assumptions.

\begin{Assum} \label{assum:detect_condi} \cite{hammouri90}. There exists a bounded time-varying matrix $K(t)$ such that the system $\dot{\eta}(t) = (A(t,y) - K(t)C(t))\eta(t)$ is globally exponentially stable \cite{Rodriguez09,besancon06}.
\end{Assum}

\begin{Assum} \label{assum:IE_condi}
\cite{zhang02a,Roy2017}.
Let $\Lambda(t)$ be the solution of $\dot{\Lambda}=\bigl(A(y)-K(t)C\bigr)\Lambda+\Psi(y)$, with $\Lambda(t_{k})=0$, and define
\begin{equation}\label{eq:insta_infomatrix}
    N(t)\triangleq\Lambda^{\top}(t)C^{\top}Q(t)C\Lambda(t)\succeq 0,
\end{equation}
as the instantaneous information matrix at time $t$, where $Q(t)$ is a bounded symmetric positive matrix, and $\|N(t)\|\leq\bar{n}<\infty$ for all $t \geq 0$. Building on $N(t)$, let $\Gamma(t)=\Gamma(t)^{\top}\succ0$ denote the \emph{(accumulated) information matrix}, generated by the least-squares law with forgetting factor $\rho_{\theta}>0$,
\begin{equation}\label{eq:gamma_law}
   \dot{\Gamma}(t)=-\rho_{\theta}\,\Gamma(t)+N(t),
   \qquad \Gamma(t_{k})\succ0,
\end{equation}
which stays symmetric positive definite for all $t \geq 0$. Then, the accumulated information on the active interval $[t_{k},t_{k+1})$ is defined as
\begin{equation}\label{eq:Accum_infomatrix}
    \mathcal{N}_{\ell}(t_{k},t)\triangleq\int_{t_{k}}^{t} N(s)\,ds .
\end{equation}
Therefore, there exist constants $\delta_{\ell}>0$ and $T_{\textrm{min},\ell}>0$ \juan{(as given by the dwell-time Definition~(\ref{def:dwell_time}))} such that
\begin{equation}\label{eq:IEcondi}
    \mathcal{N}_{\ell}\bigl(t_{k},\,t_{k}+T_{\textrm{min},\ell}\bigr)
    \;\succeq\;\delta_{\ell}\,I .
\end{equation}
\end{Assum}

Inequality~(\ref{eq:IEcondi}) is denoted \juan{ here as the finite-time window excitation condition (motivated by \cite{kreisselmeier90,ortega26}) for the switched system~(\ref{eq:asnl}) subject to the dwell-time constraint as denoted in Definition~(\ref{def:dwell_time}). Then, the mode-wise adaptive observer $\Pi_{\ell}\in\Pi=\{\Pi_{\ell}\}_{\ell\in\mathcal{M}}$ assigned to the active subsystem $\ell$ on the interval $[t_{k},t_{k+1})$ could be expressed as}
\juan{
\begin{equation}\label{eq:adaptiveObs}
\Pi_{\ell}:\!
\begin{cases}
\begin{array}{lll}
\dot{\hat{z}} &=& A(y)\hat{z}+\beta(y)+\Psi(y)\hat{\theta}_{\ell}
   +\bigl[S^{-1}C^{\top} \\
   &&+\Lambda\Gamma^{-1}\Lambda^{\top}C^{\top}\bigr]Q\bigl(y-C\hat{z}\bigr),\\[3pt]
\dot{S} &=& -\rho_{z}\,S-A(y)^{\top}S-S\,A(y)+C^{\top}QC,\\[3pt]
\dot{\Lambda} &=& \bigl(A(y)-S^{-1}C^{\top}QC\bigr)\Lambda+\Psi(y),\\[3pt]
\dot{\Gamma} &=& -\rho_{\theta}\,\Gamma+\Lambda^{\top}C^{\top}QC\,\Lambda,\\[3pt]
\dot{\hat{\theta}}_{\ell} &=& \Gamma^{-1}\,\Lambda^{\top}C^{\top}Q\bigl(y-C\hat{z}\bigr),
\end{array}
\end{cases}
\end{equation}
}
\juan{
with $S(t_{k})\succ0$, $\Lambda(t_{k})=0$, and $\Gamma(t_{k})\succ0$. Here, $\hat{z}\in\mathbb{R}^{n_{x}}$ and $\hat{\theta}_{\ell}\in\mathbb{R}^{n_{p}}$ are the state and parameter estimates of mode~$\ell$, respectively, and $Q(t)=Q(t)^{\top}\succ0$ is a bounded weighting matrix. The gain matrix $K(t)$ is set as $K(t)=S^{-1}(t)C^{\top}Q(t)$, with $S(t)=S(t)^{\top}\succ0$ and $\rho_{z}>0$ as a forgetting factor. Under the Lemma~(\ref{Lemm:obs_f_ell}) the solution of $\dot{S}$ in (\ref{eq:adaptiveObs}) remains bounded and uniformly positive definite, so that Assumption~(\ref{assum:detect_condi}) is fulfilled. The auxiliary filter $\Lambda(t)$ and the information matrix $\Gamma(t)$ (generated by the least-squares law with a forgetting factor $\rho_{\theta}>0$) are those introduced in Assumption~(\ref{assum:IE_condi}). While subsystem~$\ell$ is active, $\Pi_{\ell}$ integrates~(\ref{eq:adaptiveObs}) from the initial conditions $\bigl(\hat{z}(t_{k}),\hat{\theta}_{\ell}(t_{k}),S(t_{k}),\Lambda(t_{k}),\Gamma(t_{k})\bigr)$, whereas every inactive observer $\Pi_{\ell'}$, $\ell'\neq\ell$, holds its internal state frozen until mode~$\ell'$ becomes active again. Then, the following theorem holds for a selected switched mode~$\ell$.
}

\begin{theo} \label{theo:modewise}
Consider a switched system~(\ref{eq:asnl}) that admits, for every mode $\ell\in\mathcal{M}$, the state-affine form $\Sigma_{\ell}$ in~(\ref{eq:stateaffine}); let Assumptions~(\ref{assum:detect_condi})-(\ref{assum:IE_condi}) hold. Fix a realization $\sigma(t)=\ell$ on the active interval $[t_{k},t_{k+1})$ and let $\Gamma(t)$ be the information matrix~(\ref{eq:gamma_law}) of Assumption~(\ref{assum:IE_condi}). \juan{Then the mode-wise adaptive observer $\Pi_{\ell}\in\Pi=\{\Pi_{\ell}\}_{\ell\in\mathcal{M}}$ given by~(\ref{eq:adaptiveObs}), with parameter error $\tilde{\theta}_{\ell}=\hat{\theta}_{\ell}-\theta_{\ell}$ and state error $\tilde{z}=\hat{z}-z$, has the following properties on $[t_{k},t_{k+1})$:}
\begin{enumerate}[label=(\roman*),leftmargin=2.2em]

\item \emph{(\textbf{\juan{Lyapunov dissipation}})} The weighted error energy $V_{\ell}(t)=\tilde{\theta}_{\ell}^{\top}\Gamma\tilde{\theta}_{\ell}$ is non-increasing along the excitation-driven dynamics; equivalently, $\lVert\tilde{\theta}_{\ell}(t)\rVert_{\Gamma}$ does not grow while mode $\ell$ is active.
 
\item \emph{(\textbf{\juan{Contractive parameter estimation}})} If the active time satisfies the dwell-time constraint $\mathrm{AT}_{\ell}\geq T_{\textrm{min},\ell}$, then
\begin{equation}\label{eq:theta_error_contraction}
   \bigl\lVert\tilde{\theta}_{\ell}(t_{k+1})\bigr\rVert_{\Gamma}
   \;\leq\;\rho_{\ell}\,
   \bigl\lVert\tilde{\theta}_{\ell}(t_{k})\bigr\rVert_{\Gamma},
   ~ \text{with}~\rho_{\ell}\in(0,1),
\end{equation}

where $\rho_{\ell}$ is non-increasing in $\mathrm{AT}_{\ell}$; \juan{in particular, a longer active interval yields a stronger contraction.}

\item \emph{(\textbf{\juan{Bounded state error}})} The state error satisfies $\lVert\tilde{z}(t)\rVert\leq\epsilon_{z}$ for all $t\in[t_{k},t_{k+1})$.
\end{enumerate}
The results hold for arbitrary initial conditions $z(t_{k}),\hat{z}(t_{k}),\hat{\theta}_{\ell}(t_{k})$ and any $\theta_{\ell}\in\mathbb{R}^{n_{p}}$.
\end{theo}


\begin{pf}
We omitted time dependency for the sake of simplicity. Suitably replacing the expression of $\dot{\hat{\theta}}_{\ell}$ into $\dot{\hat{z}}$ in~(\ref{eq:adaptiveObs})
yields the expression,
\begin{equation}\label{eq:zhat_compact}
   \dot{\hat{z}}=A(y)\hat{z}+\beta(y)+\Psi(y)\hat{\theta}_{\ell}
   +K\bigl(y-C\hat{z}\bigr)+\Lambda\dot{\hat{\theta}}_{\ell}.
\end{equation}
As $\tilde{z}=\hat{z}-z$, $\tilde{\theta}_{\ell}=\hat{\theta}_{\ell}-\theta_{\ell}$,
and $\dot{\theta}_{\ell}=0$ on $[t_{k},t_{k+1})$ (Definition~\ref{def:sw_signal}), the error dynamics read
\begin{align}
   \dot{\tilde{z}} &= \bigl(A(y)-KC\bigr)\tilde{z}+\Psi(y)\tilde{\theta}_{\ell}+\Lambda\dot{\tilde{\theta}}_{\ell},
   \label{eq:ztilde}\\
   \dot{\tilde{\theta}}_{\ell} &= -\,\Gamma^{-1}\,\Lambda^{\top}C^{\top}QC\,\tilde{z}.
   \label{eq:thetatilde}
\end{align}
\underline{\juan{\textit{Step 1}}}. Consider a linear combination of $\tilde{z}$ and $\tilde{\theta}_{\ell}$, as defined by 
\juan{
\begin{equation}\label{eq:linear_combi}
   \eta\triangleq\tilde{z}-\Lambda\tilde{\theta}_{\ell},
\end{equation}
}
which, by differentiating it and using~(\ref{eq:ztilde}) together with the filter $\dot{\Lambda}$ defined in Assumption~(\ref{assum:IE_condi}),
\juan{
\begin{align}
   \dot{\eta}
   &=\dot{\tilde{z}}-\dot{\Lambda}\tilde{\theta}_{\ell}-\Lambda\dot{\tilde{\theta}}_{\ell}
    =\bigl(A(y)-KC\bigr)\tilde{z}+\Psi(y)\tilde{\theta}_{\ell}-\dot{\Lambda}\tilde{\theta}_{\ell}\nonumber\\
   &=\bigl(A(y)-KC\bigr)\tilde{z}-\bigl(A(y)-KC\bigr)\Lambda\tilde{\theta}_{\ell}\nonumber\\
   &=\bigl(A(y)-KC\bigr)\eta.
   \label{eq:eta_dyn}
\end{align}
}
\juan{Hence, the dynamic of~(\ref{eq:linear_combi})} depends only on the pair $\bigl(A(y),C\bigr)$ and the gain $K$, \juan{that is to say}, it is independent of $\sigma(t)$ and $\theta_{\ell}$ and evolves for all $t\geq0$. By Assumption~(\ref{assum:detect_condi}) the origin is globally exponentially stable, \juan{ i.e., there exist $c_{0}$, and $\lambda>0$ for which 
\begin{equation} \label{eq:eta_vanish}
\lVert\eta(t)\rVert\leq c_{0}e^{-\lambda(t-t_{k})}\lVert\eta(t_{k})\rVert,
\end{equation}
is satisfies during the active time interval $[t_{k},t_{k+1})$.
}

\smallskip
\underline{\juan{\textit{Step 2}}}. \juan{Solving $\tilde{z}$ from~(\ref{eq:linear_combi}), replacing it in~(\ref{eq:thetatilde}) and writing $N=\Lambda^{\top}C^{\top}QC\Lambda\succeq0$ as in (\ref{eq:insta_infomatrix}), yields}
\begin{equation}\label{eq:theta_cascade}
   \dot{\tilde{\theta}}_{\ell}
   = -\,\Gamma^{-1} N\,\tilde{\theta}_{\ell}
     \;-\;\Gamma^{-1}\Lambda^{\top}C^{\top}QC\,\eta.
\end{equation}
\juan{By \textit{Step 1} result}, the second term $\;\Gamma^{-1}\Lambda^{\top}C^{\top}QC\,\eta$ is a vanishing source. Now, consider the candidate Lyapunov functional
\begin{equation}\label{eq:Lyap_candidate}
   V_{\ell}(\tilde{\theta}_{\ell})=\tilde{\theta}_{\ell}^{\top}\Gamma\,\tilde{\theta}_{\ell}.
\end{equation}
Since $\Gamma(t)=\Gamma(t)^{\top}\succ0$ for every $t$, so is its inverse $\Gamma^{-1}(t)$, and the quadratic form~(\ref{eq:Lyap_candidate}) vanishes only at $\tilde{\theta}_{\ell}=0$. \juan{Moreover, $\Gamma(t)$ is the solution of~(\ref{eq:gamma_law}) driven by the continuous, uniformly bounded matrix $N(t)$, so that $\Gamma(t)$ is continuously differentiable on $[t_{k},t_{k+1})$ and so is $V_{\ell}$. The forgetting-factor $\rho_{\theta}$ keeps $\Gamma(t)$ uniformly bounded above, with $\lVert N(t)\rVert\leq\bar{n}$ (Assumption~\ref{assum:IE_condi}), integration of~(\ref{eq:gamma_law}) yields}
\begin{equation}
\Gamma(t)\preceq e^{-\rho_{\theta}(t-t_{k})}\Gamma(t_{k})+\tfrac{\bar{n}}{\rho_{\theta}}I\preceq\bar{p}\,I,
\end{equation}
\juan{where $\bar{p}$ provides the uniform upper bound,} while the accumulated excitation
in~(\ref{eq:IEcondi}) provides the uniform lower bound $\Gamma(t)\succeq \underaccent{\bar}{p}\,I\succ0$ on the active interval. Hence, $V_{\ell}$ is bounded by
\begin{equation}\label{eq:Lyap_bounds}
   \underaccent{\bar}{p}\,\lVert\tilde{\theta}_{\ell}\rVert^{2}
   \;\leq\; V_{\ell}(\tilde{\theta}_{\ell}) \;\leq\;
   \bar{p}\,\lVert\tilde{\theta}_{\ell}\rVert^{2},
   ~~~ 0<\underaccent{\bar}{p}\leq\bar{p}<\infty .
\end{equation}
\juan{Therefore, $V_{\ell}(t)$ is indeed a candidate Lyapunov functional, which is continuously differentiable on $[t_{k},t_{k+1})$}. Differentiating $V_{\ell}$ along the homogeneous part of~(\ref{eq:theta_cascade}), i.e., $\dot{\tilde{\theta}}_{\ell}=-\Gamma^{-1} N\tilde{\theta}_{\ell}$, yields
\begin{align}\label{eq:Vdot}
   \dot{V}_{\ell}
   &=2\,\tilde{\theta}_{\ell}^{\top}\Gamma\,\dot{\tilde{\theta}}_{\ell}
     +\tilde{\theta}_{\ell}^{\top}\dot{\Gamma}\,\tilde{\theta}_{\ell}\nonumber\\
   &=-2\,\tilde{\theta}_{\ell}^{\top}N\,\tilde{\theta}_{\ell}
     +\tilde{\theta}_{\ell}^{\top}\bigl(-\rho_{\theta}\Gamma+N\bigr)\tilde{\theta}_{\ell}\nonumber\\
   &=-\rho_{\theta}\,V_{\ell}-\tilde{\theta}_{\ell}^{\top}N\,\tilde{\theta}_{\ell}
   =-\rho_{\theta}\,V_{\ell}-\bigl\lVert Q^{1/2}C\Lambda\tilde{\theta}_{\ell}\bigr\rVert^{2}.
\end{align}
Since $N\succeq0$ and $\rho_{\theta}>0$, the derivative~(\ref{eq:Vdot}) is negative definite, $\dot{V}_{\ell}\leq-\rho_{\theta}V_{\ell}\leq0$, which proves~(i): $V_{\ell}$, and hence $\lVert\tilde{\theta}_{\ell}\rVert_{\Gamma}$, is non-increasing while mode~$\ell$ is active.
 
\smallskip
 
\underline{\juan{\textit{Step 3}}}. The differential inequality, $\dot{V}_{\ell}\leq-\rho_{\theta}V_{\ell}$, obtained in~(\ref{eq:Vdot}) \juan{integrates over the active interval $[t_{k},t_{k+1})$, as}
\begin{equation}\label{eq:V_balance}
   V_{\ell}(t_{k+1})\;\leq\;e^{-\rho_{\theta}(t_{k+1}-t_{k})}\,V_{\ell}(t_{k})
   \;=\;e^{-\rho_{\theta}\mathrm{AT}_{\ell}}\,V_{\ell}(t_{k}).
\end{equation}
Recalling $V_{\ell}=\lVert\tilde{\theta}_{\ell}\rVert_{\Gamma}^{2}$, yields the contraction in the parameter estimation error as,
\begin{equation} \label{eq:proof_theta_error_contraction}
   \bigl\lVert\tilde{\theta}_{\ell}(t_{k+1})\bigr\rVert_{\Gamma}
   \;\leq\;e^{-\rho_{\theta}\mathrm{AT}_{\ell}/2}\,
   \bigl\lVert\tilde{\theta}_{\ell}(t_{k})\bigr\rVert_{\Gamma}, 
\end{equation}
this is exactly~(\ref{eq:theta_error_contraction}) with 
\begin{equation} \label{eq:rho_ell}
\rho_{\ell}=e^{-\rho_{\theta}\mathrm{AT}_{\ell}/2}\in(0,1).
\end{equation}
\juan{For a given $\rho_{\theta}$ and an active time interval satisfying the dwell-time constraint ($\mathrm{AT}_{\ell}\geq T_{\textrm{min},\ell}$ as stated in Definition~(\ref{def:dwell_time})) the value of $\rho_{\ell}$ is maintained below one.} \juan{The role of the finite-time window excitation condition~(\ref{eq:IEcondi}) in Assumption~(\ref{assum:IE_condi}) is to
keeps $N(t)$ from collapsing under $\rho_{\theta}$ in~(\ref{eq:gamma_law}), thus securing the uniform lower bound $\Gamma(t)\succeq\underaccent{\bar}{p}\,I$ used in~(\ref{eq:Lyap_bounds})}. Finally, the full parameter error dynamics~(\ref{eq:theta_cascade}) is an uniformly exponentially stable homogeneous system driven by the exponential decaying signal $\eta$, hence is itself uniformly exponentially stable \cite{Loria2002}, which preserves $\rho_{\ell}<1$ in the presence of the vanishing forcing. This proves~(ii). Moreover, from~(\ref{eq:eta_vanish}) and~(\ref{eq:rho_ell}) is clear that a longer active interval attains a stronger contraction, as $\eta$ vanishes according to ~(\ref{eq:eta_vanish}) and $\rho_{\ell}$ becomes smaller as stated by ~(\ref{eq:rho_ell}), in agreement with~(ii) statement.
 
\smallskip

\underline{\juan{\textit{Step 4}}}. \juan{
Solving $\tilde{z}$ from~(\ref{eq:linear_combi}), results in $\tilde{z}=\eta+\Lambda\tilde{\theta}_{\ell}$. Considering $\Lambda$ bounded, i.e., $\bar{\Lambda}\triangleq\sup_{t}\lVert\Lambda(t)\rVert<\infty$ (Assumption~\ref{assum:detect_condi}), along with ~(\ref{eq:eta_vanish}) produces
\begin{equation}\label{eq:z_bound}
   \lVert\tilde{z}(t)\rVert\leq\lVert\eta(t)\rVert+\bar{\Lambda}\,\lVert\tilde{\theta}_{\ell}(t)\rVert,
   \qquad \forall t\in[t_{k},t_{k+1}).
\end{equation}
Taking into account the results of \textit{Step 1} ($\eta$ vanishing according to~(\ref{eq:eta_vanish})), and \textit{Step 3} ($\tilde{\theta}_{\ell}$ is contractive according to~(\ref{eq:proof_theta_error_contraction})), thus, $\lVert\tilde{z}(t)\rVert\leq \epsilon_{z}$ for $t\in[t_{k},t_{k+1})$, where
\begin{equation} \label{eq:epsilon}
\epsilon_{z}=c_{0}e^{-\lambda \mathrm{AT}_{\ell}}\lVert\eta(t_{k})\rVert+\bar{\Lambda}\, e^{-\rho_{\theta}\mathrm{AT}_{\ell}/2}\,
   \bigl\lVert\tilde{\theta}_{\ell}(t_{k})\bigr\rVert_{\Gamma}.
\end{equation}
Since $\lambda,\rho_{\theta}>0$, $\mathrm{AT}_{\ell}\geq T_{\textrm{min},\ell}$; as well as $\eta$, $\Lambda$, and $\tilde{\theta}_{\ell}$ are bounded, therefore $\tilde{z}$ is bounded. This proves~(iii).
} \qed 
\end{pf} \\[3pt]

\juan{
The imposition of the dwell time constraint is required to fulfill with~(\ref{eq:IEcondi}) in Assumption~(\ref{assum:IE_condi}). Accordingly, the dwell time term could be expressed as,
\begin{equation}\label{eq:dwelltime_definition}
   T_{\textrm{min},\ell}
   \;=\;\inf\Bigl\{\,T>0 \;\bigm|\;
   \lambda_{\textrm{min}}\!\bigl(\mathcal{N}_{\ell}(t_{k},t_{k}+T)\bigr)\geq\delta_{\ell}\Bigr\},
\end{equation}
and any $\mathrm{AT}_{\ell}\geq T_{\textrm{min},\ell}$ ensure the strict contraction property stated in Theorem~(\ref{theo:modewise}).(ii). 
}

\juan{
\begin{rem}[Setting the dwell time]\label{rem:set_Tmin}
Finding the solution of Equation~(\ref{eq:dwelltime_definition}) is computationally feasible: simulating the filter $\Lambda$ and accumulating $\mathcal{N}_{\ell}$, the dwell time term $T_{\textrm{min},\ell}$ is the first time $\lambda_{\textrm{min}}(\mathcal{N}_{\ell})$ crosses $\delta_{\ell}$, thus validating the Inequality~(\ref{eq:IEcondi}) in Assumption~(\ref{assum:IE_condi}). Accordingly, we check numerically Equation~(\ref{eq:dwelltime_definition}) for the two numerical examples presented in the next section.
\end{rem}
}

\section{Numerical Examples} \label{s:numerical_examples}

A mode-wise adaptive observer as presented in~(\ref{eq:adaptiveObs}) could serve not only for state estimation but also for adaptive control formulations that improve current predictive control approaches for switched nonlinear systems in a wide range of biomedical control applications \cite{sereno25}. \juan{Here, we exemplify the mode-wise adaptive observer~(\ref{eq:adaptiveObs}) on two autonomous switched nonlinear systems: The Hindmarsh-Rose model \cite{hindmarshrose84}, which stands for a fast dynamic process, and an HIV within-host model \cite{hernandez21a}, which stand for a slow dynamics process.
}

\subsection{\juan{Hindmarsh-Rose model}}

\juan{Consider that the spiking-bursting behavior of the membrane potential of a single neuron could be described by the following set of nonlinear ordinary differential equations \cite{hindmarshrose84},
\begin{align}\nonumber
\dot{x}_{1} &= x_{2} - a_{\sigma}\,x_{1}^{3} + b_{\sigma}\,x_{1}^{2} - x_{3} + I_{\mathrm{a}},\\ \nonumber
\dot{x}_{2} &= c - d_{\sigma}\,x_{1}^{2} - x_{2},\\
\dot{x}_{3} &= r\big(s(x_{1}-x_{\mathrm{R}}) - x_{3}\big). 
\label{eq:hr}
\end{align}
Where $x_{1}$ stands for the membrane potential of the neuron, $x_{2}$ stands for the fast recovery (or spiking) variable, which lumps the transport of fast ions across the membrane, and $x_{3}$ stands for the slow adaptation (or bursting) current, which accounts for the slow ionic exchange responsible for the alternation between quiescent and firing phases. As the HR model is dimensionless, here the interpretation of the temporal evolution is made in time units, setting this format for the interpretation of our active interval~(\ref{def:Active_time}) and the dwell-time concepts~(\ref{def:dwell_time}). The HR model can be represented by the class of autonomous switched nonlinear system described by~(\ref{eq:asnl}), where the dynamics are driven solely by the switching signal $\sigma(t)=\ell$, in this case $\ell\in\mathcal{M}=\{1,~2,~3\}$, where each mode $\ell$ stands for a distinct firing regime of the neuron.}

\juan{
HR model parameters are described as follows. The term $I_{\mathrm{a}}$ represents the applied current that sustains oscillation, while $c$ denotes the resting level of the fast recovery variable. The parameter $r$ sets the time-scale separation between the slow adaptation current and the fast subsystem and $s$ stands for the sensitivity of the adaptation current to deviations of the membrane potential from its resting value $x_{\mathrm{R}}$. The coefficients $a_{\sigma(t)}$ and $b_{\sigma(t)}$ shape the cubic and quadratic nonlinearity of the fast current-voltage characteristic, thus governing the amplitude and the width of the spikes under the firing regime selected by $\sigma(t)$. Similarly, $d_{\sigma(t)}$ denotes the strength with which the membrane potential recruits the fast recovery variable under the firing regime selected by $\sigma(t)$. Accordingly, the switched unknown parameter vector of mode~$\ell$ is $\theta_{\ell}=(a_{\ell},~b_{\ell},~d_{\ell})^{\top}\in\mathbb{R}^{n_{p}}$, with $n_{p}=3$, whereas the non-switched parameters are fixed as $c=1$, $s=4$, $x_{\mathrm{R}}$, $r=6\times10^{-3}$, and $I=3.2$.}

\smallskip
\juan{Considering the HR state vector as $z(t) = (x_{1},x_{2},x_{3})^{\top}$ and the measured output $y=x_{1}$, HR model~(\ref{eq:hr}) could be written in the state-affine form $\Sigma_{\ell}$ of~(\ref{eq:stateaffine}), where 
\begin{equation*}
A=\begin{bmatrix}0&1&-1\\0&-1&0\\ r s&0&-r\end{bmatrix},\;
\beta(y)=\begin{bmatrix}I_{\mathrm{a}}\\ c\\ -r s\,y\end{bmatrix},\;
\end{equation*}
\begin{equation*}
\Psi(y)=\begin{bmatrix}-y^{3}&y^{2}&0\\0&0&-y^{2}\\0&0&0\end{bmatrix},\;
C=\begin{bmatrix}1&0&0\end{bmatrix},
\end{equation*}
Note that in this case, $A$ entries do not depend on the output, but all his values are known, and the whole regime dependence of~(\ref{eq:hr}) is concentrated in the term $\Psi(y)\theta_{\ell}$. Since, $(A,\beta,C)$ are common to the three firing regimes and only $\theta_{\ell}$ changes with $\sigma(t)$, the three realizations $\Sigma_{1},\Sigma_{2},\Sigma_{3}$ could be represented by the switched state-affine representation $\Sigma_{\sigma}$ of~(\ref{eq:swtiched_state_affine}) by simply replacing $\theta_{\ell}$ with $\theta_{\sigma(t)}$. Regarding the assumptions of Theorem~\ref{theo:modewise}, the pair $(A,C)$ is observable ($\det\mathcal{O}=r-1\neq0$), so Lemma~(\ref{Lemm:obs_f_ell}) holds and the solution $S(t)$ of the Riccati-like filter in~(\ref{eq:adaptiveObs}) stays bounded and uniformly positive definite, which fulfills Assumption~(\ref{assum:detect_condi}) with $K(t)=S^{-1}(t)C^{\top}Q$; in turn, the regressors $y^{3}$ and $y^{2}$ are functionally independent under the bursting excitation, so Lemma~(\ref{Lemm:ident_switched}) holds and $\theta_{\ell}$ is locally identifiable from the measurable output.}

\smallskip
\juan{For the numerical simulation, a periodic-cycle switching signal, with an active time interval $\mathrm{AT}_{\ell}=400$~units, is considered. The unknown switched parameters are set as $\theta_{1}=(1,3,5)$, $\theta_{2}=(0.95,2.6,4.7)$, and $\theta_{3}=(0.88,3.3,4)$. The family of mode-wise adaptive observers is defined as $\Pi=\{\Pi_{1},~\Pi_{2},~\Pi_{3}\}$, where $\Pi_{1}$, $\Pi_{2}$, and $\Pi_{3}$ account for the firing regimes $1$, $2$, and $3$, respectively. 
Each $\Pi_{\ell}$ is a realization of~(\ref{eq:adaptiveObs}), and is integrated only while $\sigma(t)=\ell$, from its initial activation $t_{k}$ as $S(t_{k})= 5 \times I_{3}$, $\Lambda(t_{k})= 0$, and $\Gamma(t_{k})=I_{3}$, while the two inactive observers hold their internal state frozen until their regime is revisited. 
The design constants in~(\ref{eq:adaptiveObs}) are set as $Q=5$, $\rho_{z}=8$, and $\rho_{\theta}=0.03$ for all $\Pi_{\ell}$, with $\ell\in\{1,~2,~3\}$, and the estimates are initialized $30\%$ away from their true values.
}

\begin{figure}
\centering
\includegraphics[width=1.0\columnwidth]{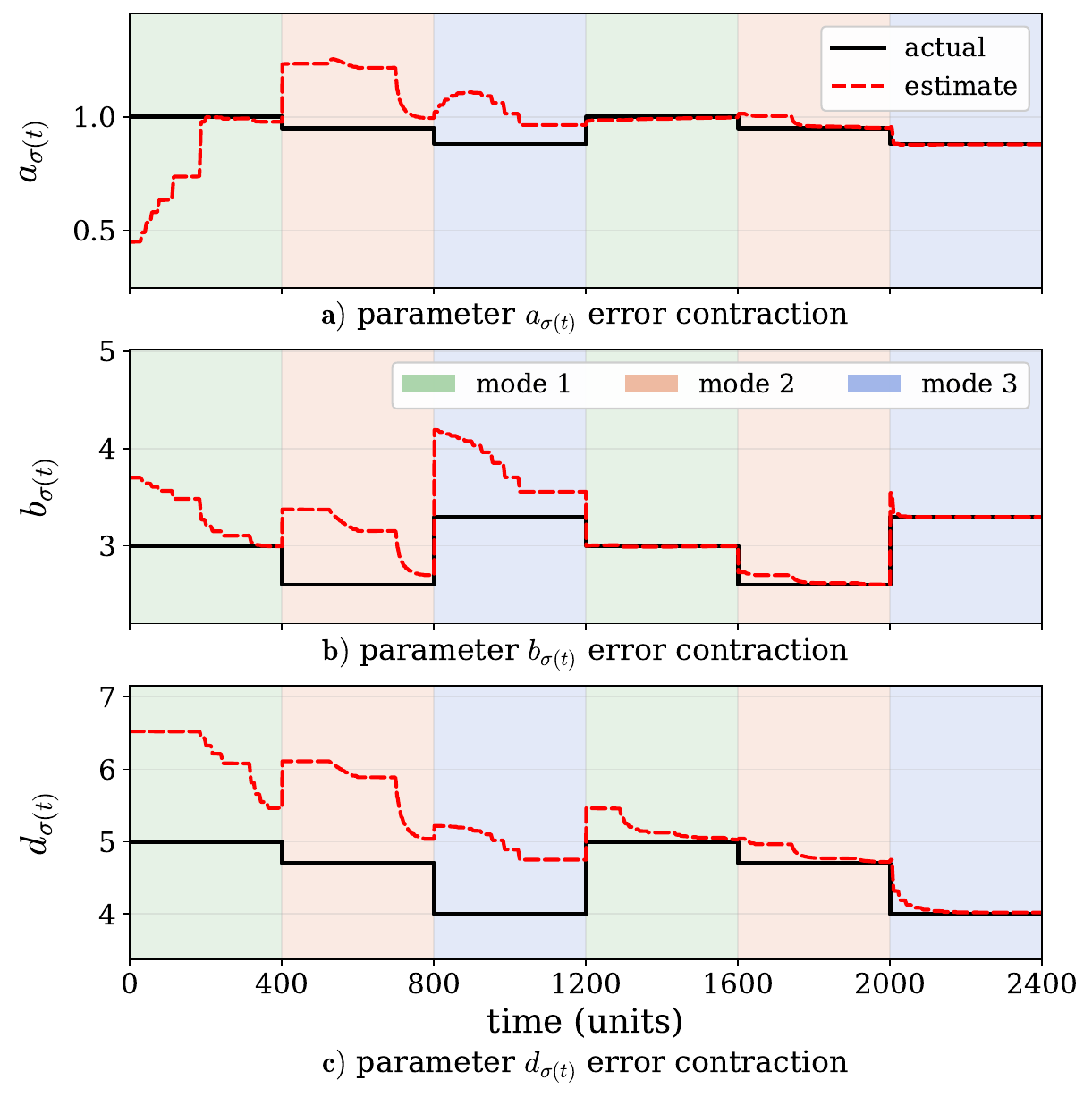}
\caption{\juan{HR model: Estimates for each switched parameter. From top to bottom, a) parameter $a_{\sigma(t)}$, b) parameter $b_{\sigma(t)}$, and c) parameter $d_{\sigma(t)}$. Black solid line: actual value; red dashed line: active-mode estimate. Shaded bands indicate the actual active mode, $1$ (light green), $2$ (light orange), $3$ (light blue).}}
\label{fig:hr_conv}
\end{figure}

\juan{For the HR model, Figure~\ref{fig:hr_conv} shows the estimation of the switched parameter vector $\theta_{\sigma(t)}=(a_{\sigma(t)},b_{\sigma(t)},d_{\sigma(t)})^{\top}$ produced by the parameter update law presented in our proposal of mode-wise adaptive observer~(\ref{eq:adaptiveObs}). Starting from different initial conditions in both the states and parameters, each estimate is refined only while its own mode is active, whereas the two inactive observers freeze their last estimate until their mode is revisited. Since every mode is re-activated along the time cycle, its second activation starts from the value retained at the previous visit and therefore improves further its estimates, illustrating how the mode-wise architecture accumulates information across successive visits of the same mode. Our proposal achieve a final per-mode relative error that remain below $7\times10^{-2}$ for all the switched modes.}

\juan{For the HR model, Figure~\ref{fig:hr_states} shows the state estimation, the per-mode normalized parameter estimation error $\lVert\tilde{\theta}_{\ell}\rVert/\lVert\theta_{\ell}\rVert$, and the switching signal realization $\sigma(t)$ over the simulation period, from $t=0$ to $t=2400$ units. As observed in Figure~\ref{fig:hr_states}, panels (a), (b), and (c), our proposal of mode-wise adaptive observer~(\ref{eq:adaptiveObs}) is capable of reconstructing the bursting trajectories of the unmeasured fast recovery variable $y$ and the slow adaptation current $z$ from the single output $x$, in agreement with the bounded state estimation error presented in Theorem~\ref{theo:modewise}.(iii). Moreover, in Figure~\ref{fig:hr_states}.(d) one can observer how the per-mode parameter estimation error decreases during active windows, while is held constant when the mode is inactive.}

\begin{figure}[h]
\centering
\includegraphics[width=1.0\columnwidth]{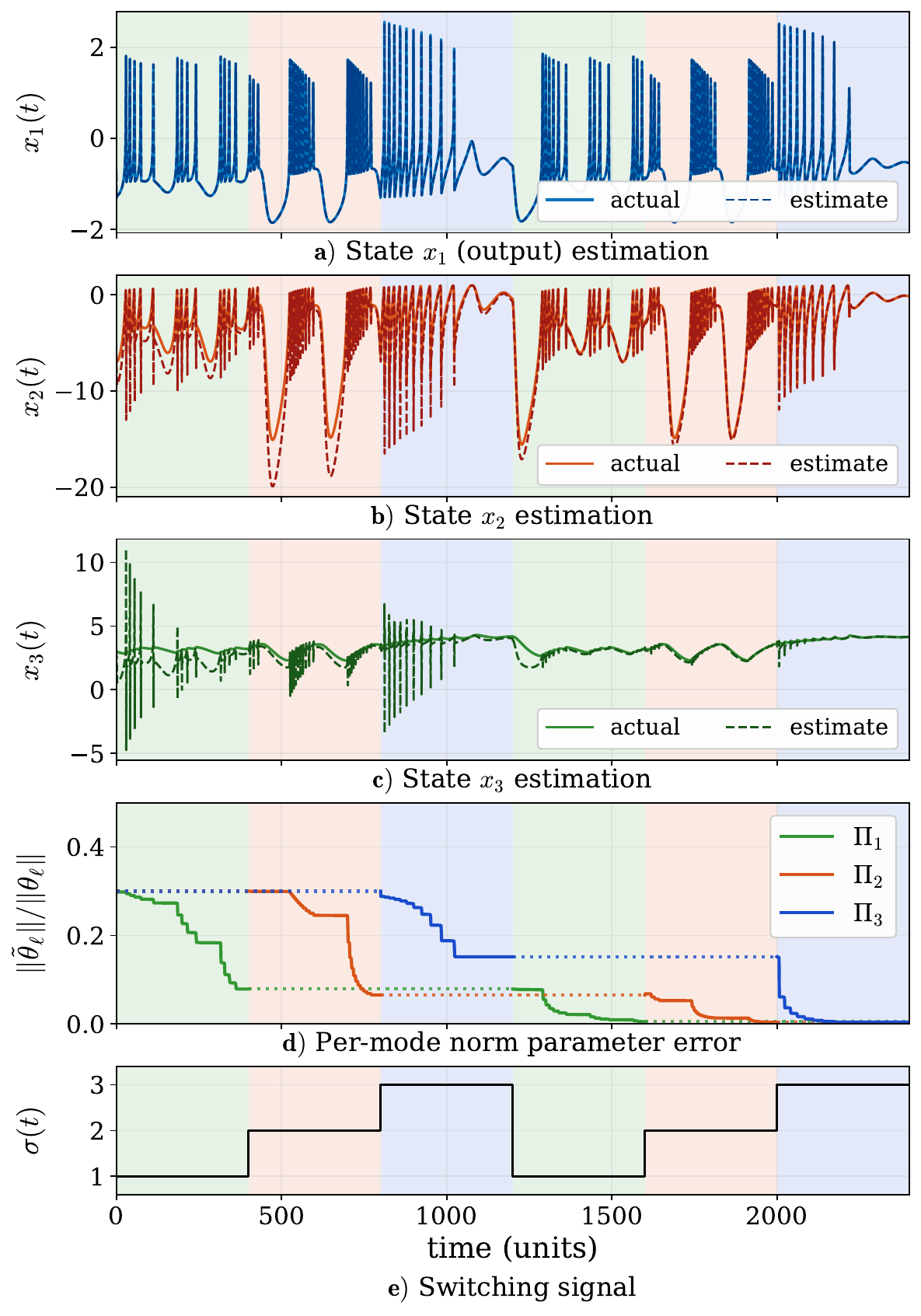}
\caption{\juan{HR model: simulation summary. From top to bottom, a) the membrane potential $x$, b) the fast recovery variable $y$, c) the slow adaptation current $z$ (actual in solid-line vs.\ estimate in dash-line), d) the per-mode normalized parameter error (modes: $1$ (green), $2$ (orange), and $3$ (blue), solid-line while the mode is active, dotted-line while frozen), and e) the switching signal $\sigma(t)$.}}
\label{fig:hr_states}
\end{figure}

\begin{figure*}[ht]
\centering
\includegraphics[width=1.0\textwidth]{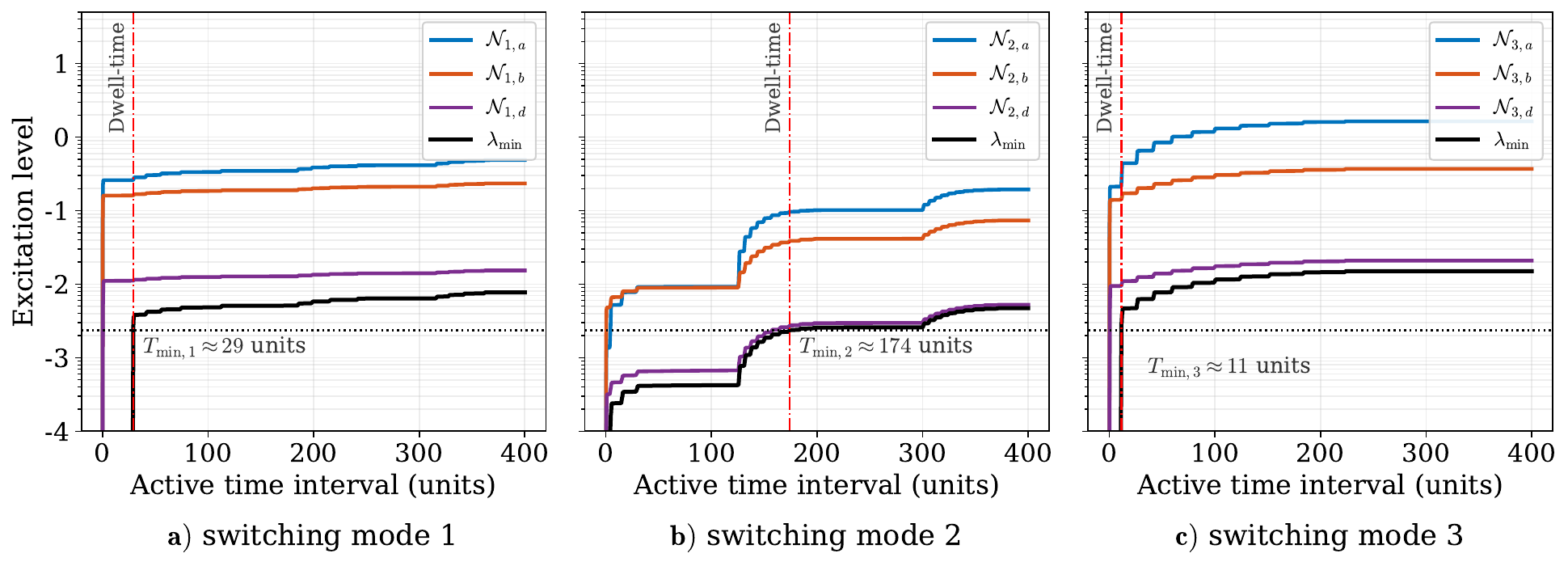}
\caption{\juan{HR model: Excitation condition~(\ref{eq:IEcondi}) and dwell-time constraint~(\ref{eq:dwelltime_definition}) validation. From left to right, a) switching mode $1$ (firing regime $1$), b) switching mode $2$ (firing regime $2$), and c) switching mode $3$ (firing regime $3$).} Panels (a)-(c) are in log scale ($0$ means $10^{0}$). For each mode, the accumulated per-element excitation $\mathcal{N}_{\ell,jj}$ ($j\in\{a,b,d\}$) and the smallest eigenvalue $\lambda_{\min}$ of the accumulated information are shown; the threshold $\delta_{\ell}$ (horizontal dotted black line) is reached at the minimal dwell time $T_{\textrm{min},\ell}$ (vertical dash-dotted red line).}
\label{fig:hr_pe}
\end{figure*}

\juan{For the HR model, Figure~\ref{fig:hr_pe} summarized the numerical analysis of the finite-time window excitation condition presented in the Inequality~(\ref{eq:IEcondi}), and as stated in Remark~(\ref{rem:set_Tmin}), for the computation of the minimal dwell time term $T_{\textrm{min},\ell}$. The accumulated information $\mathcal{N}_{\ell}(t_{k},t)$ of~(\ref{eq:Accum_infomatrix}) is built from the filter $\Lambda$ of~(\ref{eq:adaptiveObs}) with $\Lambda(t_{k})=0$. 
The neuronal bursts drive the membrane potential through a wide amplitude range, so the regressor $\Psi(y)$ is persistently exciting, and all three parameters are identifiable in every mode. 
Accordingly, the smallest eigenvalue of the accumulated information matrix crosses the threshold $\delta_{\ell}$ of condition~(\ref{eq:IEcondi}) at minimal dwell times $T_{\textrm{min},\ell}\approx29,\,174,\,11$~units, for modes $1$, $2$, and $3$, respectively, all computed according to Remark~(\ref{rem:set_Tmin}). The larger value for mode~$2$ reflects a slower accumulation of information in that firing regime. In all three cases $T_{\textrm{min},\ell}$ stays well below the active time interval, which is set $\mathrm{AT}_{\ell}=400$~units for all the switching modes, so Assumption~(\ref{assum:IE_condi}) and the dwell-time constraint $\mathrm{AT}_{\ell}\ge T_{\textrm{min},\ell}$ of Definition~(\ref{def:dwell_time}) are satisfied.}

\subsection{HIV \juan{within-host} model}

Consider that the HIV genotypes and immune response dynamics can be represented with a set of ordinary differential equations of the following form \cite{hernandez13s},
\begin{align} \nonumber
\dot{T} &= s_T + \frac{\rho_T}{C_T + V_T} T V_T - \sum_{i=1}^{n} k_{T,\sigma}^i T V_i - \delta_T T  \\ \nonumber
\dot{M} &= s_M + \frac{\rho_M}{C_M + V_T} M V_T - \sum_{i=1}^{n} k_{M,\sigma}^i M V_i - \delta_M M \\ \nonumber
\dot{T}_i^* &= k_{T,\sigma}^i T V_i + \sum_{j=1}^{n} \mu m_{i,j} V_j T - \delta_{T^*} T_i^* \\ \nonumber
\dot{M}_i^* &= k_{M,\sigma}^i M V_i + \sum_{j=1}^{n} \mu m_{i,j} V_j M - \delta_{M^*} M_i^* \\
\dot{V}_i &= p_{T,\sigma}^i T_i^* + p_{M,\sigma}^i M_i^* - \delta_V V_i. \label{eq:hiv_model}
\end{align}
Where $T$ and $T_{i}^{*}$, stand for the population of healthy and infected CD4$+~$T cells, respectively. Similarly, $M$ and $M_{i}^{*}$, stand for the population of healthy and infected macrophages, respectively. Each virus strain is represented by $V_{i}$. As a simple motivating example of HIV mutation, we consider the aforementioned model with $4$ genetic variants, i.e., $i \in \{1,~2,~ 3,~ 4\}$, and under pressure of two possible drug therapies; hence, the total viral load is given by $V_T = \sum_{i=1}^{4} V_i$, and $\sigma(t) = \ell$, with $\ell \in \mathcal{M} = \{0, ~1,~2\}$, where, for this particular example, we set $\ell = 0$ as the no-therapy or free virus evolving scenario. Similar to the HR model, the HIV mutation model~(\ref{eq:hiv_model}) falls under to the autonomous switched nonlinear system described by~(\ref{eq:asnl}), where each switching mode $\ell$ stands for a distinct drug therapy regime for the patient.

HIV model~(\ref{eq:hiv_model}) dynamics are characterized by a set of non-therapy-dependent and therapy-dependent parameters. Hence, a subset of model parameters does not depend on the switching signal realization and is described as follows. For the healthy CD4$+~$T cells and macrophages, the source term productions are represented with $s_T$ and $s_M$; the clearance rates are $\delta_T$ and $\delta_M$. While $\delta_{T^*}$ and $\delta_{M^*}$ denote the death rates of their infected counterparts, and $\delta_V$ represents the clearance rate of free virus particles. The terms $\rho_T$ and $\rho_M$, represent the maximum proliferation rates of CD4$+~$T cells and macrophages induced by viral load, along with their respective half-saturation constants $C_T$ and $C_M$. While the subset switched parameters, they are described as follows. The infection rates are represented by $k_{T,\sigma}^i$ and $k_{M,\sigma}^i$ at which CD4$+~$T cells and macrophages, respectively, become infected by the genotype $i$ under therapy selected by $\sigma(t)$. Similarly, $p_{T,\sigma}^i$ and $p_{M,\sigma}^i$ denote the viral production rates from infected CD4$+~$T cells and infected macrophages of genotype $i$ under therapy selected by $\sigma(t)$. Finally, the mutation dynamics are driven through the infected populations, in which the term $m_{i,j} = 1$ represents the mutation that occurs from genotype $j$ to $i$, $m_{i,j} = 0$ otherwise. We consider the mutation tree as stated \juan{in \cite{hernandez12}}. The wild-type variant ($V_1$) could mutate to $V_2$ and $V_3$, which implies $m_{2,1} = 1$, and $m_{3,1} = 1$, while variants $V_2$ and $V_3$, could mutate to the highly resistant genotype $V_4$, which means $m_{4,2} = 1$, and $m_{4,3} = 1$, for all other cases $m_{i,j} = 0$.

\juan{The full state vector $x\in\mathbb{R}^{14}$, could be considered as $x=(\{T_{i}^{*},M_{i}^{*},V_{i}\}_{i \in \mathbb{N}_{1:4}},T,M)^{\top}$ and the measured outputs as the viral loads of each variant $V_i$ and the healthy CD4$+~$T cells, i.e., $y=(V_{1},V_{2},V_{3},V_{4},T)^{\top}$, then $C$ is the selection matrix such that $y=Cx$. The switched unknown parameter vector is taken as the vector of infection rates of the CD4$+~$T cell compartment, $\theta_{\ell}=(k_{T,\ell}^{1},~k_{T,\ell}^{2},~k_{T,\ell}^{3},~k_{T,\ell}^{4})^{\top}\in\mathbb{R}^{n_{p}}$, with $n_{p}=4$, while the remaining therapy-dependent coefficients are assumed known for each prescribed therapy. Under these settings, the HIV model~(\ref{eq:hiv_model}) admits the state-affine form $\Sigma_{\ell}$ of~(\ref{eq:stateaffine}), in which we consider the change of coordinates $z \triangleq x / \lambda$, where $\lambda$ is a known constant scaling factor chosen for convenience of numerical computation.}

\juan{Consequently, $\dot{z}=A(y)z+\beta(y)+\Psi(y)\theta_{\ell}$, on every active interval $[t_{k},t_{k+1})$ with $\sigma(t)=\ell$, with $\ell \in \mathcal{M} = \{0, ~1,~2\}$. Products $T \cdot V_{i}$ that carry the unknown infection rates are characterized by measured variables only, so they are accumulated in the output-only regressor $\Psi(y)$, whose $i$-th column is the regressor of $k_{T,\ell}^{i}$ and enters the $\dot{T}$ and $\dot{T}^{*}_{i}$ rows; the source terms and the saturating proliferation terms are collected in $\beta(y)$; and the rest of the linear and bilinear terms (clearance, mutation, macrophage infection, and viral production) are incorporated through the output-dependent matrix $A(y)$. Since the therapy switching signal $\sigma(t)$ only changes the unknown vector $\theta_{\ell}$, the three realizations $\Sigma_{0},\Sigma_{1},\Sigma_{2}$ could be represented by the switched state-affine representation $\Sigma_{\sigma}$ of~(\ref{eq:swtiched_state_affine}). Regarding the assumptions of Theorem~\ref{theo:modewise}, observability and identifiability analysis is derived by generating the generalized observability-identifiability matrix from the output map and its Lie derivatives~(\ref{eq:Obsmatrix}), as stated in~\cite{villaverde16,villaverde19a}. We make use of the FISPO algorithm of the STRIKE-GOLDD toolbox~\cite{diazseoane23}. For every therapy realization $\ell\in\mathcal{M}$, the subsystem $f_{\ell}$ has $n_{x}=14$ states, no inputs, $n_{y}=5$ measurable outputs, and $n_{p}=4$ unknown switched parameters. Thus, by computing three Lie derivatives to build the generalized observability-identifiability matrix, its rank equals $18=n_{x}+n_{p}$, i.e., it is full rank. Hence, Lemma~(\ref{Lemm:obs_f_ell}) and Lemma~(\ref{Lemm:ident_switched}) holds; consequently, the solution $S(t)$ of the Riccati-like filter in~(\ref{eq:adaptiveObs}) stays bounded and uniformly positive definite, which fulfills Assumption~ (\ref{assum:detect_condi}) with $K(t)=S^{-1}(t)C^{\top}Q$, for all the three realizations of the switched state-affine representation~(\ref{eq:swtiched_state_affine}).
}

\begin{figure}[ht]
\centering
\includegraphics[width=1.0\columnwidth]{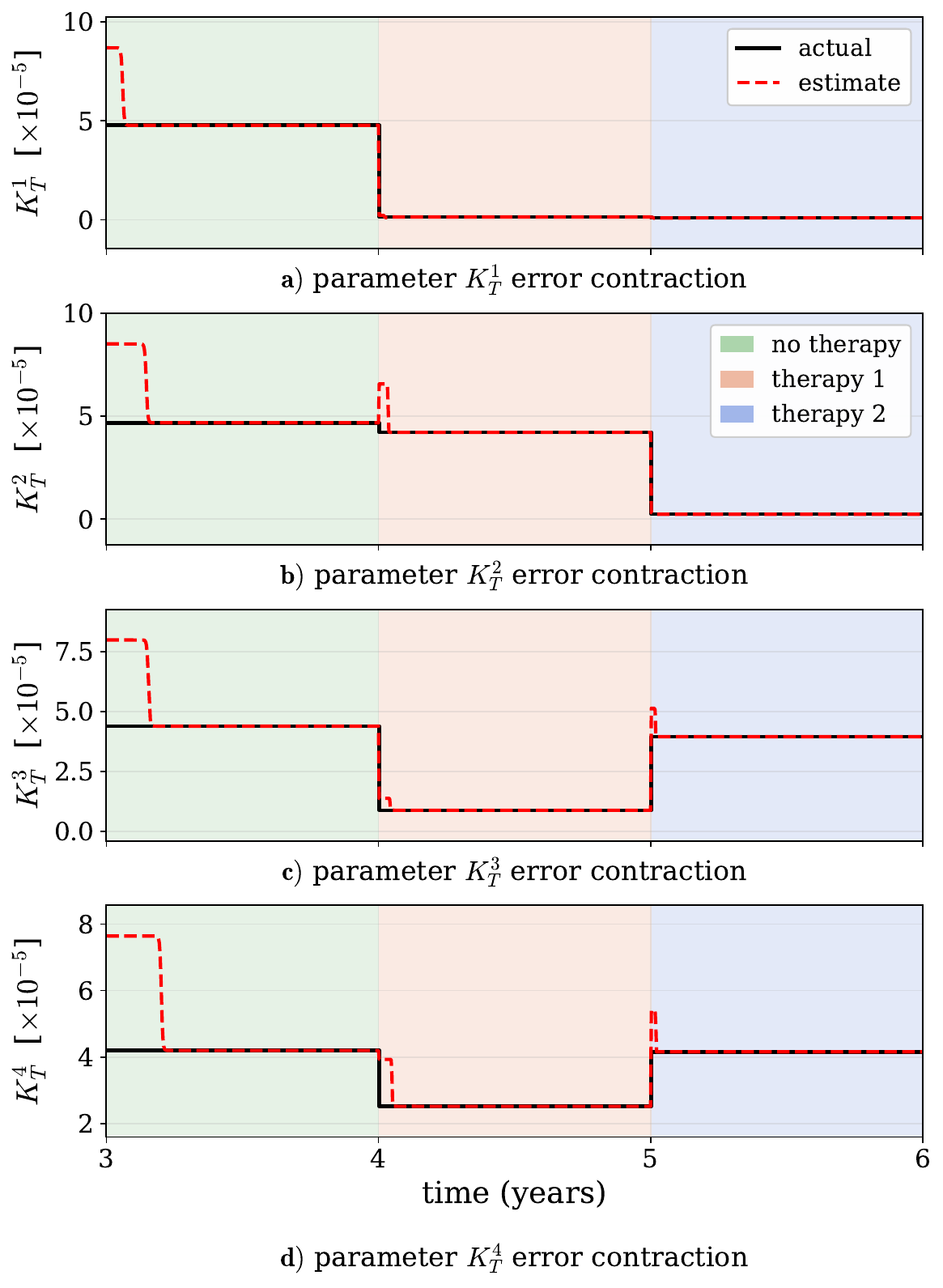}
\caption{\juan{HIV model: Estimates for each switched parameter. From top to bottom, the switched infection-rate parameters: a) $K_{T}^{1}$, b) $K_{T}^{2}$, c) $K_{T}^{3}$, d) $K_{T}^{4}$. Black solid line: actual value, red dashed line: active-mode estimate. The shaded bands indicate the active drug therapy (no therapy: light green; drug therapy~1: light orange; drug therapy~2: light blue).}}
\label{fig:KTconv}
\end{figure}

\begin{figure}[ht]
\centering
\includegraphics[width=1.0\columnwidth]{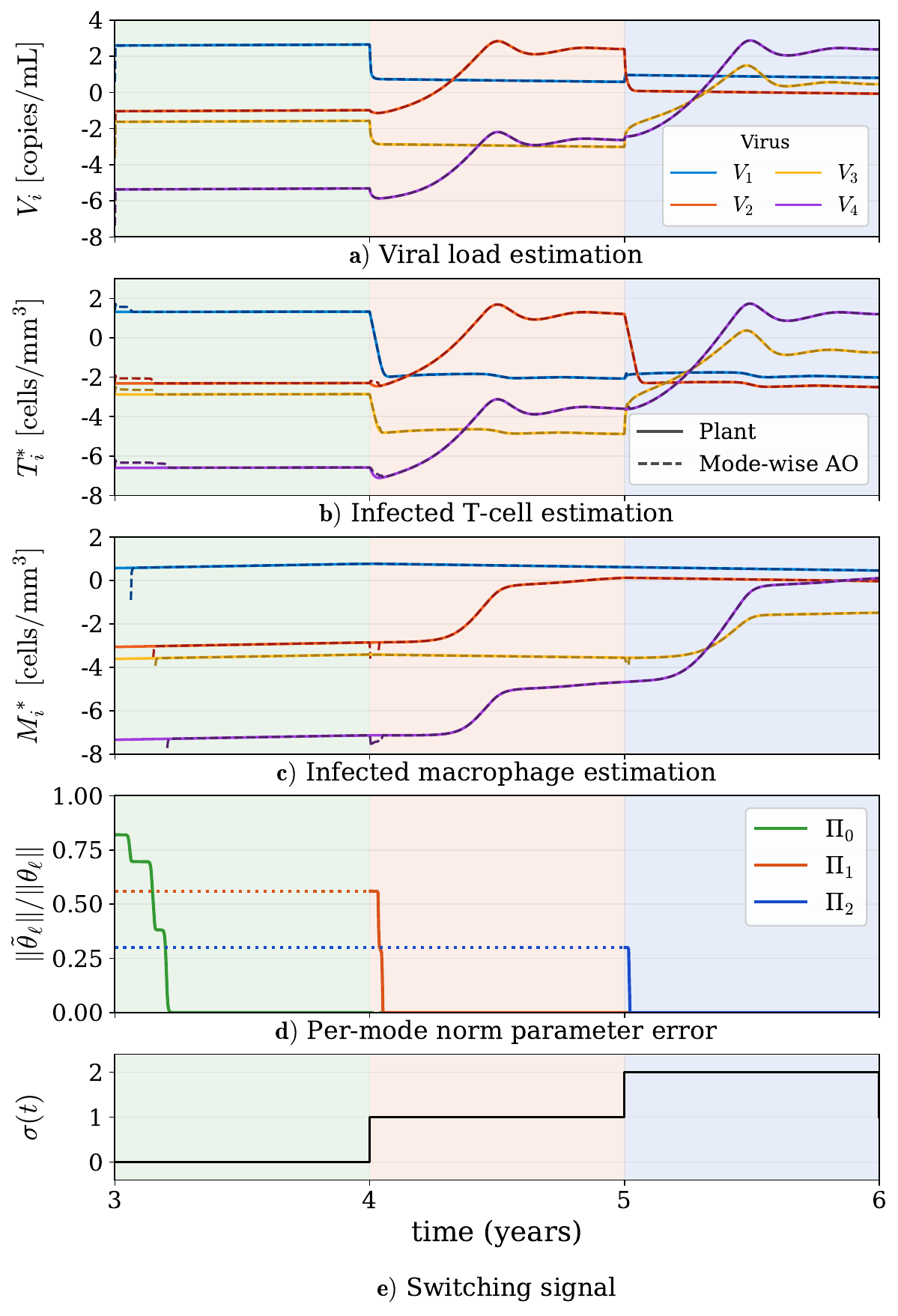}
\caption{\juan{HIV model: simulation summary. From top to bottom, a) viral loads $V_{i}$, b) infected CD4$^{+}$ T cells $T_{i}^{*}$, and c) infected macrophages $M_{i}^{*}$. Panels (a)-(c) are in log scale ($0$ means $10^{0}$) and solid lines: actual trajectories, dashed lines: active-observer estimates. panel d) the per-mode normalized parameter error $\lVert\tilde{\theta}_{\ell}\rVert/\lVert\theta_{\ell}\rVert$ (solid line while active, dotted line while frozen; $\Pi_{0}$ green, $\Pi_{1}$ orange, $\Pi_{2}$ blue); and in panel e) the therapy switching signal $\sigma(t)$.}}
\label{fig:states}
\end{figure}

\begin{figure*}[ht]
\centering
\includegraphics[width=1.0\textwidth]{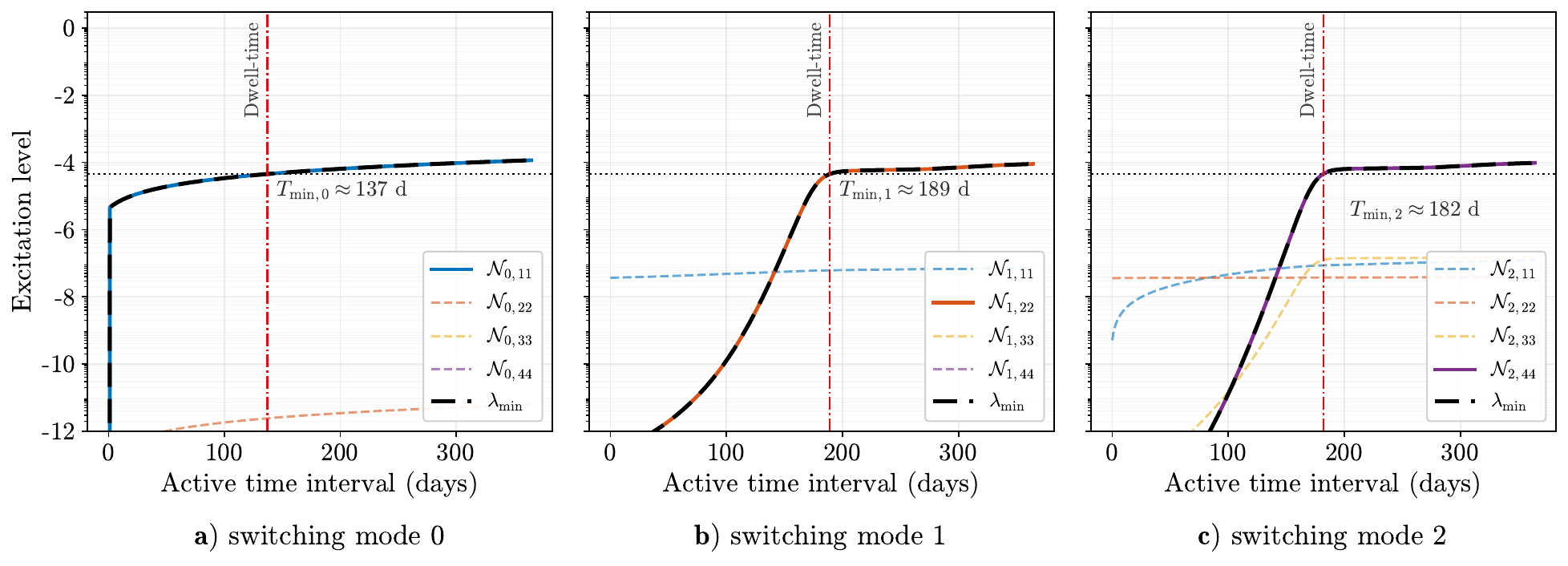}
\caption{\juan{HIV model: Excitation condition~(\ref{eq:IEcondi}) and dwell-time constraint~(\ref{eq:dwelltime_definition}) validation. From left to right, a) switching mode $0$ (no-therapy), b) switching mode $1$ (therapy 1), and c) switching mode $2$ (therapy 2).} Panels (a)-(c) are in log scale ($0$ means $10^{0}$). For each mode~$\ell$, the accumulated per-element excitation $\mathcal{N}_{\ell,jj}(t_{k},t)$ of the infection rates $k_{T,\ell}^{j}$ (solid: excited; dashed: weakly excited) is shown together with the smallest eigenvalue $\lambda_{\min}$ of the accumulated information matrix on the excited subspace (black). The excitation threshold $\delta_{\ell}$ (horizontal dotted black line) is reached at the minimal active time $T_{\textrm{min},\ell}$ (vertical dash-dotted red line).}
\label{fig:hiv_pe}
\end{figure*}

For the numerical simulation, we consider the initial simulation time as $t_{0} = 0$, and during the first $4$ years, the virus and immune response dynamics evolve under no drug therapy conditions. At $t=4$ years, a SWATCH treatment is considered; the logic behind this strategy is the alternation of drug therapies, in such a case, we consider an alternation time window of $1$ year for each drug. Thus, from $0$ to $4$ years, no drug therapy is applied, then $\ell=0$. At $t=4$ years, drug therapy 1 starts, $\ell=1$, until $t=5$ years, when treatment switches to drug therapy 2, $\ell=2$, and then the SWATCH treatment continues until the simulation ends. Accordingly, the active time of every therapy mode is $\mathrm{AT}_{\ell}=365$ days. The family of mode-wise adaptive observers is defined as $\Pi = \{ \Pi_0,~\Pi_1,~\Pi_2 \}$, where $\Pi_0$, $\Pi_1$, and $\Pi_2$, account for no drug therapy, drug therapy 1, and drug therapy 2, respectively. \juan{Each $\Pi_{\ell}$ is the realization of~(\ref{eq:adaptiveObs}) driven by the pair $\bigl(A(y),C\bigr)$ and the regressor $\Psi(y)$ described above, integrated only while its therapy is active, from $S(t_{k})= 1 \times I_{14}$, $\Lambda(t_{k})=1 \times \mathbf{1}_{14 \times4}$, and $\Gamma(t_{k})= 1 \times 10^{5} \cdot I_{4}$, while the two inactive observers hold their internal state frozen until their therapy is prescribed again. The forgetting factor of the state filter is set as $\rho_{z}=10$, and the forgetting factor of the information matrix~(\ref{eq:gamma_law}) is chosen as $\rho_{\theta}=1,~4,~8$ for $\Pi_0$, $\Pi_1$, and $\Pi_2$, respectively, whereas $Q = 5 \times I_{5}$, and the state estimates are initialized $30\%$ away from their true values, while the switched parameters estimates are initialized $80\%$, $60\%$, and $30\%$ away from their true values for $\ell=0$, $\ell=1$, and $\ell=2$, respectively. Finally, the scaling factor is set as $\lambda=1 \times 10^{8}$.}

\juan{For the HIV model, Figure~\ref{fig:KTconv} shows the estimation of the switched parameter vector $\theta_{\sigma(t)}=(k_{T,\sigma(t)}^{1},~k_{T,\sigma(t)}^{2},~k_{T,\sigma(t)}^{3},~k_{T,\sigma(t)}^{4})^{\top}$, with each of the switched signal realization representing a particular drug therapy scenario. Similarly to the HR model, the adaptive observer~(\ref{eq:adaptiveObs}) its started from different initial conditions, and each switched parameter estimate is refine only while its own drug therapy mode is active. For the HIV model case, one can observe that just one visit to each drug therapy achieve an strong parameter estimation error contraction, i.e., in all cases the relative estimation error falls below $10^{-3}$ within the corresponding active interval for each drug therapy mode.
}

\juan{For the HIV model, Figure~\ref{fig:states} summarized the numerical simulation, where the state estimation, the per-mode normalized parameter estimation error, and the switching signal are shown. As observed in Figure~\ref{fig:states}, panels (a), (b), and (c), the mode-wise adaptive observer~(\ref{eq:adaptiveObs}) is capable of reconstructing the unmeasured infected CD$4+~$ T cells, and infected macrophages trajectories from the measurable viral loads, and healthy CD$4+~$ T cell measurements. Moreover, in Figure~\ref{fig:states}, panel (d), one can observe how the normalized parameter error of each drug therapy mode decreases while it is active and is held constant while frozen. This behavior aligns with the Lyapunov dissipation of Theorem~\ref{theo:modewise}.(i) and with the freeze-and-resume mechanism of the mode-wise architecture, which in turns confirms that the mode-wise architecture removes the zero-input disturbance that the switching behavior would otherwise inject into the parameter estimation dynamics. Finally, in Figure~\ref{fig:states}, panel (e), one can observe the full drug therapy scheduling from the switching signal realization over the simulation period. Qualitatively, the numerical simulations shows a root-mean-squared logarithmic errors of the active-mode estimates over the whole horizon that remain below $3\times10^{-7}$ for the four infection rates $K_{T}^{i}$ and below $1\times10^{-2}$ for the viral loads and the infected cell populations, showing an accurate joint state-parameter estimation.}

\juan{For the HIV model, Figure~\ref{fig:hiv_pe} summarized the numerical analysis of the excitation condition~(\ref{eq:IEcondi}) and the dwell-time constraint~(\ref{eq:dwelltime_definition}), according to Remark~(\ref{rem:set_Tmin}). The accumulated information $\mathcal{N}_{\ell}(t_{k},t)$ of~(\ref{eq:Accum_infomatrix}) is built from the filter $\Lambda$ of the adaptive observer formulation~(\ref{eq:adaptiveObs}) with $\Lambda(t_{k})=0$. Since the $j$-th column of $\Lambda$ is the regressor of the infection rate $k_{T,\ell}^{j}$, whose energy is proportional to the viral load $V_{j}$ of genotype $j$, and each therapy renders a different genotype dominant, the excitation is naturally assessed in an element-wise manner. For every mode $\ell$ we report the per-element accumulated excitation $\mathcal{N}_{\ell,jj}(t_{k},t)$, identify the excited subset $\mathcal{S}_{\ell}=\{\,j:\mathcal{N}_{\ell,jj}\ \text{non-negligible}\,\}$, and verify the matrix condition~(\ref{eq:IEcondi}) restricted to $\mathcal{S}_{\ell}$, i.e. $\lambda_{\min}\!\bigl(\mathcal{N}_{\ell}[\mathcal{S}_{\ell}]\bigr)\geq\delta_{\ell}$, whose first crossing yields the minimal active time $T_{\textrm{min},\ell}$ of~(\ref{eq:dwelltime_definition}), computed as prescribed by Remark~(\ref{rem:set_Tmin}). As shown in  Figure~\ref{fig:hiv_pe}, each therapy reshapes the genotype distribution and thus energizes the accumulated information matrix~(\ref{eq:Accum_infomatrix}). In every case, the smallest eigenvalue of the accumulated information on the excited subspace crosses the threshold $\delta_{\ell}$ inside the active interval window, with $T_{\textrm{min},0}\approx179$, $T_{\textrm{min},1}\approx189$, and $T_{\textrm{min},2}\approx184$ days, all below $\mathrm{AT}_{\ell}=365$ days; hence, the dwell-time constraint $\mathrm{AT}_{\ell}\geq T_{\textrm{min},\ell}$ of Definition~(\ref{def:dwell_time}) is satisfy for every mode. Notably, the resulting time windows admits a clinical reading in the digital twin sense: roughly six months of sustained therapy are enough for the twin to recalibrate the infection rates of the genotypes under each particular drug therapy.
}

\section{Conclusion} \label{s:conclusion}

In this paper we have explored a mode-wise design for the joint estimation of states and switched parameters in autonomous switched nonlinear systems. Rather than relying on a single observer structure, a dedicated adaptive observer is tailored to each switching mode and is active only during the corresponding active-time interval, which removes the zero-input disturbance that the switching behavior would otherwise inject into the parameter estimation dynamics. For every mode, local nonlinear observability and local structural identifiability were established, and a dwell-time condition on the active time was derived from a finite-window persistence-of-excitation requirement, guaranteeing the contractiveness of the parameter estimation error. \juan{Moreover, for every switched mode, a direct connection between the length of the active-time interval and the contractiveness of the parameter estimation error was established.} 

The approach is \juan{shown} first with an illustrative numerical example and subsequently on an HIV within-host model with immune response dynamics under a SWATCH therapy schedule. The numerical results show that the mode-wise adaptive observer is capable of accurately reconstructing the infected state trajectories and each of the infection-rate parameters, with the required minimal dwell time remaining well within the active-time intervals. \juan{For the digital twin perspective, the established dwell-time condition plays an important role, as it defines how long a drug therapy must be held for the twin to recalibrate itself before the next switching instant. Further work would be required to shorten this minimal time required for recalibration and therefore enable faster adaptive structures.}

\juan{Online adaptation is one of the fundamentals in the digital twin sphere. The computational model needs to be able to actively track the ever-evolving dynamics of its physical counterpart~\cite{Laubenbacher22,wright20}. Besides, for many of these real-world digital twin applications, we do not need exact convergence but boundedness, as there is inherent parameter uncertainty and stochasticity that is outside the existing mathematical modeling knowledge of biological systems. As it was mentioned at the introduction, current medical digital twin proposals still lean on batch-based identification, which in practice postpones the model updating and, in consequence, the reliability of the model \cite{laubenbacher24n,kovatchev25,cappon25}. We consider that adaptive observer algorithms hold the potential to naturally overcome this online adaptation issue. Their inherent capability to reconstruct hidden states and adjust model parameters, jointly and online, allow them to keep the digital twin synchronized to the physical system as its dynamics evolve.}

Future work will focus on relaxing the persistence-of-excitation requirement towards initial-excitation conditions, on the measurement-aware estimation of weakly detectable genotypes, and on the integration of the mode-wise observer with switching model predictive control formulations \cite{sereno25} for output-feedback evolutionary therapy design. The extension to nonlinearly parameterized switched systems is also a promising research direction.

\section{Acknowledgements}

\juan{This material is based on work supported by the National Science Foundation under Grant DMS-2439054.
}

\bibliographystyle{unsrt}
\bibliography{My_library.bib}

\end{document}